\documentclass[11pt]{article}
\usepackage{times,fullpage}
\usepackage{algorithmic,algorithm,psfrag,graphicx}
\usepackage{bm}
\usepackage{color}

\newcommand{\trsp}{^T}
\newcommand{\tendsto}{\rightarrow}
\newcommand{\opnorm}[1]{\left\|#1\right\|_2}
\newcommand{\Frobnorm}[1]{\left\|#1\right\|_F}
\newcommand{\norm}[1]{\left\|#1\right\|}
\newcommand{\id}{\textrm{Id}}
\newcommand{\eps}{\varepsilon}

\newcommand{\lo}[1]{\textrm{o}\left(#1\right)}
\newcommand{\vTildeEps}{\tilde{v}_{\eps}}
\usepackage{amsmath,amsthm,cite,amssymb}

\newtheorem{theorem}{Theorem}
\newtheorem{corollary}[theorem]{Corollary}
\title{Second order accurate distributed eigenvector computation for extremely large matrices}
\date{First version: May 2009\\ This version: Feb 2010}
\author{Noureddine El Karoui\thanks{Department of Statistics, U.C. Berkeley, Berkeley, CA 94720. \texttt{nkaroui@stat.berkeley.edu}. Support from an Alfred P. Sloan research Fellowship and NSF grants DMS-0605169 and DMS-0847647 (CAREER) is gratefully acknowledged.}
\and
Alexandre d'Aspremont\thanks{ORFE, Princeton University, Princeton, NJ 08544. \texttt{aspremon@princeton.edu}. Support from NSF grants DMS-0625352, SES-0835550 (CDI), CMMI-0844795 (CAREER), a Peek junior faculty fellowship, a Howard B. Wentz Jr. award and a gift from Google is gratefully acknowledged.}
}

\newtheorem{assumption}{Assumption}

\newcommand{\BEAS}{\begin{eqnarray*}}
\newcommand{\EEAS}{\end{eqnarray*}}
\newcommand{\BEA}{\begin{eqnarray}}
\newcommand{\EEA}{\end{eqnarray}}
\newcommand{\BEQ}{\begin{equation}}
\newcommand{\EEQ}{\end{equation}}
\newcommand{\BIT}{\begin{itemize}}
\newcommand{\EIT}{\end{itemize}}
\newcommand{\BNUM}{\begin{enumerate}}
\newcommand{\ENUM}{\end{enumerate}}
\newcommand{\BMI}{\begin{minipage}}
\newcommand{\EMI}{\end{minipage}}

\newcommand{\BA}{\begin{array}}
\newcommand{\EA}{\end{array}}
\newcommand{\BC}{\begin{center}}
\newcommand{\EC}{\end{center}}

\newcommand{\ones}{\mathbf 1}

\newcommand{\reals}{{\mbox{\bf R}}}

\newcommand{\symm}{{\mbox{\bf S}}}  

\newcommand{\Rank}{\mathop{\bf Rank}}

\newcommand{\Card}{\mathop{\bf Card}}

\newcommand{\NumRank}{\mathop{\bf NumRank}}

\newcommand{\Expect}{\mathop{\bf E{}}}

\newcommand{\var}{\mathop{\bf var}}

\begin{document}
\maketitle
\begin{abstract}
We propose a second-order accurate method to estimate the eigenvectors of extremely large matrices thereby addressing a problem of relevance to statisticians working in the analysis of very large datasets. More specifically, we show that averaging eigenvectors of randomly subsampled matrices efficiently approximates the true eigenvectors of the original matrix under certain conditions on the incoherence of the spectral decomposition. This incoherence assumption is typically milder than those made in matrix completion and allows eigenvectors to be sparse. We discuss applications to spectral methods in dimensionality reduction and information retrieval.
\end{abstract}

\section{Introduction}
Spectral methods have a long list of applications in statistics and machine learning. Beyond dimensionality reduction techniques such as PCA or CCA \cite{anderson03,mardiakentbibby}, they have been used in clustering \cite{Ng02}, ranking \& information retrieval \cite{Page98, Hast01,Lang05} or classification for example. Computationally, one of the most attractive features of these methods is their low numerical cost, in particular on problems where the data matrix is sparse (e.g. graph clustering or information retrieval). Computing a few leading eigenvalues and eigenvectors of a matrix, using the power or Lanczos methods for example, requires performing a sequence of matrix vector products and can be processed very efficiently. This means that when the matrix is dense and has dimension $n$, the cost of each iteration is $O(n^2)$ in both storage and flops.

However, for extremely large scale problems arising in statistics or information retrieval for example, this cost quickly becomes prohibitively high and makes spectral methods impractical. In this paper, we propose a randomized, distributed algorithm to estimate eigenvectors (and eigenvalues) which makes spectral methods tractable on very large scale matrices. We show that our method is second order accurate and illustrate its performance on a few realistic datasets.

Going back to the numerical cost of spectral methods, we see that decomposing each matrix vector product in many smaller block operations partially alleviates the complexity problem, but makes the overall process very bandwidth intensive. Decomposition techniques thus improve the {\em granularity} of iterative eigenvalue methods (i.e. require many cheaper operations instead of a single very expensive one), but at the expense of significantly higher bandwidth requirements. Here, we focus on methods that improve the granularity of large-scale eigenvalue computations while having {\em very low bandwidth} requirements, meaning that they can be fully distributed over many loosely connected machines.

The idea of using subsampling to lower the complexity of spectral methods can be traced back at least to \cite{Groh91,papa00} who described algorithms based on subsampling and random projections respectively. Explicit error estimates followed in \cite{Frie04,Drin06,Achl07} which bounded the approximation error of either elementwise or columnwise matrix subsampling procedures. On the application side, a lot of work has been focused on the Pagerank vector, and \cite{Ng01} in particular study its stability under perturbations of the network matrix. Similar techniques are applied to spectral clustering in \cite{Huan08} and both works have close connections to ours. Following the {\em Netflix} competition on collaborative filtering, a more recent stream of works \cite{Rech07,Cand08,Cand09,Kesh09} has also been focused on {\em exactly} reconstructing a low rank matrix from a small, single incoherent set of observations. Finally, more recent ``volume sampling'' results provide relative error bounds \cite{Vemp09}, but so far, the sampling probabilities required to obtain these improved error bounds remain combinatorially hard to compute.

Our work here is focused on the impact of subsampling on eigenvector approximations. First we seek to understand how far we can reduce the granularity of eigenvalue methods using subsampling, before reconstructing eigenvectors becomes impossible. This question was partially answered in \cite{Cand09,Kesh09} for matrices with low rank, incoherent spectrum, using a {\em single} subset of matrix coefficients, after solving a convex program with {\em high complexity}. Here we make much milder assumptions on matrix incoherence. In particular, we allow some eigenvectors to be {\em sparse} (while remaining incoherent on their support) and we approximate eigenvectors using {\em many} simple operations on subsampled matrices. Under certain conditions on the sampling rate which guarantee that we remain in a perturbative setting, we show that simply {\em averaging} many approximate eigenvectors obtained by subsampling reduces approximation error by an order of magnitude. 

\paragraph{Notation.} In what follows, we write $\symm_n$ the set of symmetric matrices of dimension $n$. For a matrix $X\in\reals^{m\times n}$, we write $\|X\|_F$ its Frobenius norm, $\|X\|_2$ its spectral norm, $\sigma_i(X)$ its $i$-th largest singular value and let $\|X\|_\infty=\max_{ij}|X_{ij}|$, while $\Card(X)$ is the number of nonzero coefficients in $X$. We denote by $X(i,j)$ or $X_{ij}$ its $(i,j)$-th element and by $M_i$ the $i$-th column of $M$. Here, $\circ$ denotes the Hadamard (i.e entrywise) product of matrices. When $x\in\reals^n$ is a vector, we write its Euclidean norm $\|x\|_2$ and $\|x\|_\infty$ its $\ell_\infty$ norm. We write $\ones \in\reals^n$ the vector having all entries equal to 1. Finally, $\kappa$ denotes a generic constant, whose value may change from display to display. 

\section{Subsampling}
\label{s:subsamp} We first recall the subsampling procedure in \cite{Achl07} which approximates a symmetric matrix $M\in\symm_n$ using a subset of its coefficients. The entries of $M$ are independently sampled as
\BEQ\label{eq:sampl}
{S}_{ij}=
\left\{\BA{cl}
M_{ij}/p & \mbox{with probability $p$}\\
0 & \mbox{otherwise,}
\EA\right.
\EEQ
where $p\in[0,1]$ is the sampling probability. Theorem 1.4 in \cite{Achl07} shows that when $n$ is large enough
\BEQ\label{eq:bound-am07}
\|M-S\|_2 \leq 4 \|M\|_\infty \sqrt{{n}/{p}},
\EEQ
holds with high probability. In what follows, we will prove a similar bound on $\|M-S\|_2$ using incoherence conditions on the spectral decomposition of $M$.

\subsection{Computational benefits}
\label{ss:compute}
Computing $k$ leading eigenvectors and eigenvalues of a symmetric matrix of dimension $n$ using iterative algorithms such as the power or Lanczos methods (see \cite[Chap. 8-9]{Golu90} for example) only requires matrix vector products, hence can be performed in $O(kn^2)$ flops when the matrix is dense. However, this cost is reduced to $O(k\Card(M))$ flops for sparse matrices $M$. Because the matrix~$S$ defined in~(\ref{eq:sampl}) has only $pn^2$ nonzero coefficients on average, the cost of computing $k$ leading eigenvalues/eigenvectors of $S$ will typically be $1/p$ times smaller than that of performing the same task on the full matrix $M$. Of course, sampling the matrix $S$ still requires $O(n^2)$ flops, but can be done in a single pass over the data and be fully distributed. In what follows, we will show that, under incoherence conditions, averaging the eigenvectors of many independently subsampled matrices produces second order accurate approximations of the original spectral decomposition. While the global computational cost of this averaging procedure may not be globally lower, it is decomposed into many much smaller computations, and is thus particularly well adapted to large clusters of simple, loosely connected machines (Amazon EC2, Hadoop, etc.).

\begin{figure}[h]
\begin{center}
\includegraphics[width=0.8 \textwidth]{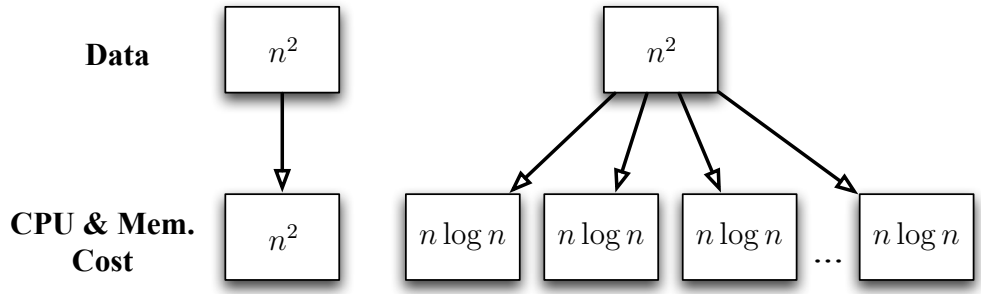}
\caption{Our objective here is to approximate the spectral decomposition problem of size $O(n^2)$ by solving many independent problems of much smaller size.}
\end{center}
\end{figure}

\subsection{Sparse matrix approximations}
\label{ss:sparse-approx}
Let us write the spectral decomposition of $M\in\symm_n$ as
\[
M=\sum_{i=1}^n \lambda_i u_iu_i^T
\]
where $u_i\in\reals^n$ for $i=1,\ldots,n$ and $\lambda\in\reals^n$ are the eigenvalues of $M$ with $\lambda_1> \ldots > \lambda_n$ (we assume they are all distinct). Let $\alpha\in[0,1]^n$, we measure the {\em incoherence} of the matrix $M$ as
\BEQ\label{eq:incoherence}
\mu(M,\alpha)=\sum_{i=1}^n |\lambda_i| n^{\alpha_i}\|u_i\|_\infty^2
\EEQ
Note that this definition is slightly different from that used in \cite{Cand09} because we do not seek to reconstruct the matrix $M$ exactly, so the tail of the spectrum can be partially neglected in our case. As we will see  below, the fact that we only seek an approximation also allows us to handle sparse eigenvectors.

Let us define a matrix $Q\in\symm_n$ with i.i.d. Bernoulli coefficients
\[
{Q}_{ij}=
\left\{\BA{cl}
1/p & \mbox{with probability $p$}\\
0 & \mbox{otherwise.}
\EA\right.
\]
We can write
\[
Q=\ones\ones^T+\sqrt{\frac{1-p}{p}} C
\]
where $C$ is has i.i.d. entries with mean zero and variance one, defined as
\[
{C}_{ij}=
\left\{\BA{cl}
\sqrt{{(1-p)}/{p}} & \mbox{with probability $p$}\\
-\sqrt{{p}/{(1-p)}} & \mbox{otherwise.}
\EA\right.
\]
We can now write the sampled matrix $S$ in (\ref{eq:sampl}) as
\BEQ\label{eq:sampling-decomp}
S=M\circ Q=M+ \sqrt{\frac{1-p}{p}} \left( \sum_{i=1}^n \lambda_i (u_iu_i^T) \circ C \right) \equiv M+E
\EEQ
and we now seek to bound the spectral norm of the residual matrix $E$ as $n$ goes to infinity. Naturally, if $\opnorm{E}$ is small, $S$ is a good approximation of $M$ in spectral terms, because of Weyl's inequality and the Davis-Kahan $\sin(\theta)$-theorem (see \cite{bhatia97}). So our aim now is to control $\opnorm{E}$ so we can guarantee the quality of spectral approximations of $M$ made using the sparse matrix $S$ which is computationally easier to work with than the dense matrix $M$. We now make the following key assumptions on the incoherence of the matrix $M$.

\begin{assumption}\label{Assump:BoundedIncoherence}
There is a sequence of vectors $\alpha^{(n)}\in[0,1]^n$ for which
\[
\mu(M,\alpha^{(n)})\leq \mu \quad \mbox{and} \quad \Card(u_i)\leq n^{\alpha^{(n)}_i}, \quad i=1,\ldots,n
\]
as $n$ goes to infinity, where $\mu$ is an absolute constant.
\end{assumption}
In what follows, we will drop the dependence of $\alpha$ on $n$ to make the notation less cumbersome, so instead of writing $\alpha^{(n)}$ we will just write $\alpha$. We have the following theorem.
\begin{theorem}\label{Thm:ControlOpNormErrorMatrix}
Suppose that Assumption \ref{Assump:BoundedIncoherence} holds. Let us call $\alpha_{min}=\min_{1\leq i \leq n}\alpha_i$. Assume that $p$ and $n$ are such that, $p<1/2$, and for a given $\delta>0$, $\alpha_{\min}>(\log n)^{(\delta-3)/4}$ and
$$
\frac{(\alpha_{\min}\log n)^{4}}{pn^{\alpha_{\min}}}\tendsto 0\;, \text{ as } n\tendsto \infty,
$$
then we have
\BEQ\label{eq:norm-bound}
\limsup_{n\tendsto \infty}\|E\|_2 \leq 2 \mu \left(pn^{\alpha_{\mathrm{min}}}\right)^{-1/2}\text{ a.s }\;.
\EEQ
\end{theorem}
\begin{proof}
Using \cite[Th.~5.5.19]{Horn91} or the fact that $uu^T\circ C =D_u C D_u$, where $D_u$ is a diagonal matrix with the vector $u$ on the diagonal (remember that $\opnorm{\cdot}$ is a matrix norm and hence sub-multiplicative), we get
\BEQ \label{eq:hadamard-ineq}
\|E\|_2 =  \sqrt{\frac{1-p}{p}} \left\| \sum_{i=1}^n \lambda_i C\circ  (u_iu_i^T)  \right\|_2 \leq \sqrt{\frac{1-p}{p}}\sum_{i=1}^n |\lambda_i| n^{{\alpha_i}/{2}}  \|u_i\|_\infty^2 \left\| \frac{C_{\alpha_i}}{n^{\alpha_{i/2}}} \right\|_2\;.
\EEQ
Since we assume that the vector $u_i$ is sparse with $\Card(u_i)\leq n^{\alpha_i}$, $C_{\alpha_i}$ is a principal submatrix of $C$ with dimension $n^{\alpha_i}$. Now,
we show in Theorem \ref{thm:ControlMaxOpNormsMatrices} (this is the key element of the proof - see p.\pageref{thm:ControlMaxOpNormsMatrices}) that
\[
\limsup_{n\tendsto \infty} \left\| \frac{C_{\alpha_i}}{n^{\alpha_{i/2}}} \right\|_2 \leq 2\;,
\]
whenever $p=o\left(\frac{(\alpha_{\min}\log n)^{4}}{n^{\alpha_{\min}}}\right)$, and $\alpha_{\min}>(\log n)^{(\delta-3)/4}$ for some $\delta>0$. (Our proof of Theorem \ref{thm:ControlMaxOpNormsMatrices} relies on a result of Vu \cite{Vu07a} and Talagrand's inequality.). This yields Equation \eqref{eq:norm-bound} and concludes the proof.
\end{proof}

The proof of the theorem makes clear that the error term coming from the sparsest eigenvector will usually dominate all the others in the residual matrix $E$.

In these approximation methods, we naturally want to use a small $p$, so that $S$ is very sparse and the computation of its spectral decomposition is numerically cheap. The result of Theorem \ref{thm:ControlOfE} guarantees that the subsampling approximation works whenever $p\gg (\alpha_{\min}\log n)^4/n^{\alpha_{\min}}$ (asymptotically, but we have in mind a very high-dimensional setting, so $n$ will be large in practice).

A natural question is therefore whether we could use $p$ much smaller than this. Separate computations (see Subsection \ref{AppSubsec:TightnessResult}) indicate that $\|C/n^{1/2}\|_2$ goes to infinity if $p\leq (\log\, n)^{1-\delta}/n$, which suggests that this subsampling approach to approximating eigenproperties of $M$ might run into trouble if the sampling rate $p$ gets smaller than $\log n/n$. As a matter of fact, we could not control the quantities $\opnorm{C_{\alpha_i}/n^{\alpha_i/2}}$ at this sampling rate, which is naturally problematic given the way we established the bound on $\opnorm{E}$. Furthermore, if the sparsest eigenvector had support disjoint from the supports of all other eigenvectors, $E$ would be the sum of two block diagonal matrices. Hence, its operator norm would be the maximum of the operator norms of the two blocks, at least one of which having potentially very large operator norm.

\subsection{Tightness}
\label{ss:tightness}


Note that, in the limit case $\alpha=\ones$ where the eigenvectors are fully dense and incoherent, our bound is similar to the original bound in \cite[Theorem 1.4]{Achl07} or that of \cite[Th~1.1]{Kesh09} (our model for $M$ is completely different however). In fact, the bounds in (\ref{eq:bound-am07}) and (\ref{eq:norm-bound}) can be directly compared. In the fully dense case where $\alpha=\ones$, we have
\[
\sqrt{n}\|M\|_\infty = \sqrt{n} \left\|\sum_{i=1}^n \lambda_i u_iu_i^T \right\|_\infty
\leq n^{-1/2} \sum_{i=1}^n |\lambda_i| n \|u_i\|_\infty^2
\leq n^{-1/2} \mu,
\]
so in this limit case, the original bound in (\ref{eq:bound-am07}) is always tighter than our bound in (\ref{eq:norm-bound}). However, in the sparse incoherent case where $\alpha\neq\ones$, the ratio of the bound (\ref{eq:bound-am07}) in \cite{Achl07} over our bound (\ref{eq:norm-bound}) becomes
\[
\frac{2\left\|\sum_{i=1}^n \lambda_i n^{\frac{(\alpha_{\mathrm{min}}+1)}{2}} u_iu_i^T \right\|_\infty}{\sum_{i=1}^n |\lambda_i| n^{\alpha_i}\|u_i\|_\infty^2},
\]
which can be large when $\alpha_{\mathrm{min}}<1$. The results in \cite{Kesh09}, which are focused on exact recovery of low rank incoherent matrices, do not apply when the eigenvectors are sparse (i.e. $\alpha\neq\ones$).

\subsection{Approximating eigenvectors}
\label{ss:approx-eigv} We now study the impact of subsampling on the eigenvectors and in particular on the one associated with the largest eigenvalue. We have the following theorem.
\begin{theorem}\label{thm:ApproxEigenvectorsOrderk+2Accurate}
Assume that the eigenvalues of $M$ are simple. Let us call $v_k\in\reals^n$ and $\lambda_{k}(S)$  the $k$-th eigenpair of $S$, and $u_k\in\reals^n$, $\lambda_k$ the $k$-th eigenpair of $M$. We write $R_k$ the reduced resolvent of $M$ associated with $u_k$, defined as
\[
R_k=\sum_{j\neq k} \frac{1}{\lambda_j-\lambda_k}u_ju_j^T,
\]
and let $\Delta_k=R_k(E-(\lambda_{k}(S)-\lambda_k)\id)$. We also call $d_k$ the separation distance of $\lambda_k$, i.e the distance from~$\lambda_k$ to the nearest eigenvalue of~$M$. If $\opnorm{E}$ satisfies $\opnorm{E}<d_k/2$, then
\begin{equation}\label{eq:ApproxNormErrorVector}
\norm{v_k-u_k+\left[\sum_{m=0}^j (-1)^m \Delta^m \right]R_kE u_k}_2\leq \frac{1}{2}\left(\frac{2\opnorm{E}}{d}\right)^{j+2} \frac{1}{1-\frac{2\opnorm{E}}{d}}
\end{equation}
having normalized $v_k$ so $v_k\trsp u_k=1$.
\end{theorem}
\begin{proof}
%
From now on we focus on $u_k$ and drop the dependence on $k$ in $u_k$, $v_k$, $R_k$, $\Delta_k$ etc... when this does not create confusion. We also use the notation $\lambda_S$ and $\lambda$ instead of $\lambda_k(S)$ and $\lambda_k$. If $v$ is normalized so that $v\trsp u=1$ (so $(v-u)\trsp u=0$), we have the explicit formula \cite[Eq.~3.29]{kato95}
$$
v-u=-(\id+R(E-\gamma \id))^{-1}REu\;,
$$
where $\gamma=\lambda_S-\lambda$. The formula is valid as soon as $(\id+R(E-\gamma \id))$ is invertible. Let us now call $\Delta=R(E-\gamma \id)$ and assume that $\Delta$ has no eigenvalues equal to -1, i.e $\id+\Delta$ is invertible. Then we have
\begin{equation}\label{eq:ExactKthOrderExpansionOfEigenvector}
v-u+\left[\sum_{m=0}^j (-1)^m \Delta^m \right]RE u=(-1)^j \Delta^{j+1}(\id+\Delta)^{-1}REu\;.
\end{equation}
We also have by construction $Ru=0$, so $REu=\Delta u$. Hence, we can write
$$
v-u+\left[\sum_{m=0}^j (-1)^m \Delta^m \right]RE u=(-1)^j \Delta^{j+2}(\id+\Delta)^{-1}u\;.
$$
Now let us call $d$ the separation distance of $\lambda$. Then $\opnorm{R}=1/d$. Our assumptions guarantee that $\opnorm{E}$ is such that $2\opnorm{E}/d<1$. We note that using Weyl's inequality, $|\lambda_S-\lambda|\leq \opnorm{S-M}=\opnorm{E}$, hence $\norm{\Delta}\leq 2\opnorm{R}\opnorm{E}=2\opnorm{E}/d$ and
$$
\opnorm{(\id+\Delta)^{-1}}\leq \frac{1}{1-\frac{2\opnorm{E}}{d}}\;.
$$
Putting all the elements together and recalling that $\norm{u}_2=1$, we get (\ref{eq:ApproxNormErrorVector}) from Equation \eqref{eq:ExactKthOrderExpansionOfEigenvector}.
\end{proof}
Spectral methods tend to focus on eigenvectors associated with extremal eigenvalues, so let us elaborate on the meaning of Theorem \ref{thm:ApproxEigenvectorsOrderk+2Accurate} for the eigenvector associated with the largest eigenvalue. If we suppose that the spectral norm of the residual matrix $E$ is smaller than half the separation distance of the largest eigenvalue, i.e
\BEQ\label{eq:pert-condition}
\|E\|_2 < (\lambda_1 - \lambda_2)/2\;,
\EEQ
the previous result (and results such as \cite[Theorem II.3.9]{kato95}) shows that we can use perturbation expansions to approximate the leading eigenvector of the subsampled matrix. Based on the bound in Equation \eqref{eq:norm-bound}, the condition stated in Equation \eqref{eq:pert-condition} will be satisfied (asymptotically with high-probability) if, for some $\eps>0$,
\[
\frac{\mu}{\sqrt{pn^{\alpha_{\mathrm{min}}}}} < (\lambda_1 - \lambda_2)/(4+\eps).
\]
We note that assumption \eqref{eq:pert-condition} is likely reasonable if one eigenvalue is very large compared to the others, which is a natural setting for methods such as PCA. (Note however that our result is not limited to the largest eigenvalue but actually applies to any eigenvalue of the original matrix $M$,~$\lambda$, for which $\opnorm{E}$ is smaller than half the distance from $\lambda$ to any other eigenvalue of $M$. In particular, the result would apply to several separated eigenvalues.) We also note that  the approximation
$$
v=u-\left[\sum_{m=0}^j (-1)^m \Delta^m \right]RE u
$$
is accurate to order $j+2$.

Let us now try to make our approximation slightly more explicit. If we write $R$ the reduced resolvent of $M$ (associated with $u_1$), and assume that $\lambda_1-\lambda_2$ stays bounded away from 0, we have in this setting, using Equation \eqref{eq:ApproxNormErrorVector} with $j=1$,
$$
v=u-REu+R(E-(\lambda_1(S)-\lambda_1)\,\id)RE u +O_P(\|E\|_2^3)\;,
$$
and therefore
\BEQ\label{eq:eigv-expansion}
v=u-REu+R(E-u^TEu\,\id)RE u +O_P(\|E\|_2^3)\;,
\EEQ
after we account for the fact that $u^TEu$ is an order-$\opnorm{E}^2$ accurate approximation of $\lambda_1(S)-\lambda_1$ \cite[Eq.~2.36 and 3.18]{kato95}. This approximation makes clear that a key component in the accuracy of our approximations will be the size of the vector $Eu$.
For simplicity here, we have normalized $v$ so that $v^Tu=1$; a similar result holds if we set $v^Tv=1$ instead, if for instance $\opnorm{E}\tendsto 0$ asymptotically.





\subsection{Second order accuracy result for eigenvectors by averaging}
In light of Equation \eqref{eq:eigv-expansion}, it is clear that $v$ is a first order accurate approximation of $u$, because of the presence of the (first-order) term $REu$ in the expansion. We now show that we can get a second order accurate approximation of the eigenvector $u$. Our results are based on an averaging procedure and hence are easy to implement in a distributed fashion. We have the following second-order accuracy result.
\begin{theorem}\label{thm:SecondOrderAccuracyAveragingProc}
Let us call $u_1$ the eigenvector associated with the largest eigenvalue of $M$, and $\nu_1=v_1/\norm{v_1}$ the eigenvector associated with the largest eigenvalue of $S$ and normalized so that $\norm{\nu_1}=1$ and $\nu_1\trsp u_1\geq 0$. Let us call $\xi=\mu/(pn^{\alpha_{\min}})^{1/2}$. Suppose that the assumptions of Theorem \ref{Thm:ControlOpNormErrorMatrix} are satisfied (hence $\xi\tendsto 0$).
Suppose also that $d=(\lambda_1-\lambda_2)$ satisfies
\begin{equation}\label{eq:pert-condition-Strong}
d\geq \xi\sqrt{\ln(\xi^{-2})}\;.
\end{equation}
Then we have
$$
\Expect\left[\norm{\nu_1-u_1}_2\right]=O\left(\frac{1}{(\lambda_1-\lambda_2)^2}\frac{\mu^2}{pn^{\alpha_{\min}}}\right)=O\left(\frac{\xi^2}{d^2}\right)\;.
$$
\end{theorem}
Practically, this means that if we average eigenvectors over many subsampled matrices (after removing indeterminacy by always making the first component positive), the residual error will be of order $\|E\|_2^2/d^2$ with
\[
\limsup_{n\tendsto \infty} \|E\|_2^2\leq 4 \frac{\mu^2}{pn^{\alpha_{\mathrm{min}}}}.
\]
In other words, by averaging subsampled eigenvectors, we gain an order of accuracy (over the method that would just take one subsampled eigenvector) by canceling the effect of the first order residual term $REu$.

\begin{proof}
To keep notations simple, we drop the index 1 in $\nu$ and $u$ in the proof (so $\nu_1=\nu$ and $u_1=u$).
In what follows, $\kappa$ is a generic constant that may change from display to display. Before we start the proof per se, let us make a few remarks.

First, there is a technical difficulty when trying to work directly with $v$, namely the fact that it appears difficult to control $\Expect\left[\opnorm{(\id+\Delta)^{-1}}\right]$ and hence to get a bound on $\Expect[\norm{v-u}]$ (with the normalization $v\trsp u=1$, $\norm{v}$ could be very large; our bounds show that this can happen with only low probability but obviously $\Expect[\norm{v}]$ could still be large). To go around this difficulty, we need two steps: first, we work with unit eigenvectors (so we go from $v$ to $\nu$), and second we need a ``regularization" step and will replace $v$ by a vector $\vTildeEps$ which is equal to $v$ with high-probability and for which we can control $\Expect[\norm{\vTildeEps-u}]$. More precisely, for $\eps>0$, we call $\vTildeEps$ the vector such that
$$
\vTildeEps=
\begin{cases}
v \text{ if } \opnorm{(\id+\Delta)^{-1}}\leq \frac{1}{\eps}\\
u-REu+\Delta RE u \text{ otherwise.}
\end{cases}
$$
Its properties are studied in Theorem \ref{thm:SecondOrderAccuracyRegularizedAveragingProc}. We call it below the $\eps$-regularized version of $v$.

We note that under the assumptions of the current theorem we have $\frac{\xi}{d}\tendsto 0$, so the results of Theorem \ref{thm:SecondOrderAccuracyRegularizedAveragingProc} apply. In particular, as shown in the proof of that Theorem, we have $\norm{M}_{\infty}^2/p^2=\lo{\xi^2}$. Also, Assumption 1 (which is made in Theorem \ref{Thm:ControlOpNormErrorMatrix}), means $\mu$ is fixed so $\xi\tendsto 0$, as $pn^{\alpha_{\min}}\tendsto \infty$.

If $v$ is the eigenvector of $S$ associated with its largest eigenvalue, using the fact that $(v-u)\trsp u=0$ by construction, we have
$$
\norm{v}_2^2=\norm{v-u}_2^2+\norm{u}_2^2=1+\norm{v-u}_2^2
$$
hence
$$
\nu=\frac{v}{\sqrt{1+\norm{v-u}_2^2}}\;.
$$
Turning our attention to $\vTildeEps$, we see that, since $Ru=0$ by construction and $R$ is symmetric,  $u\trsp \Delta=0$, so $(\tilde{v}_{\eps}-u)\trsp u=0$, and hence
$$
\norm{\tilde{v}_{\eps}}_2^2=1+\norm{\tilde{v}_{\eps}-u}_2^2\;.
$$
Now let us call
$$
\beta=\frac{\tilde{v}_{\eps}}{\sqrt{1+\norm{\tilde{v}_{\eps}-u}_2^2}}\;,
$$
we see that $\beta=\nu$ as long as $\opnorm{(\id+\Delta)^{-1}}\leq 1/\eps$, since when this happens, $v=\tilde{v}_{\eps}$. Now we have
\begin{align*}
\Expect[\norm{u-\nu}_2]&=\Expect[\norm{u-\nu}_2 1_{\nu=\beta}]+\Expect[\norm{u-\nu}_2 1_{\nu\neq \beta}]\\
&\leq \Expect[\norm{u-\beta}_2 1_{\nu=\beta}]+\Expect[\norm{u-\nu}_2 1_{\nu\neq \beta}]\\
&\leq \Expect[\norm{u-\beta}_2]+2 P(\nu\neq \beta)\;,
\end{align*}
since $\norm{u-\nu}_2\leq \norm{u}_2+\norm{\nu}_2=2$ (note the importance of the change of normalization here, as this bound would not hold with $v$ instead of $\nu$). Let us now work on controlling both these quantities. For reasons that will be clear later, we now take $\eps=2\xi/d$.
\paragraph{Control of $\bm{\Expect[\norm{u-\beta}_2].}$} Given that $u-\beta=(u-\vTildeEps)/\sqrt{1+\norm{u-\vTildeEps}_2^2}+u(1-1/\sqrt{1+\norm{u-\vTildeEps}_2^2})$, we have
\begin{align*}
\norm{u-\beta}_2&\leq \frac{\norm{u-\vTildeEps}_2}{\sqrt{1+\norm{u-\vTildeEps}_2^2}}+\norm{u}_2\left(1-\frac{1}{\sqrt{1+\norm{u-\vTildeEps}_2^2}}\right)\\
&\leq \norm{u-\vTildeEps}_2 + (\sqrt{1+\norm{u-\vTildeEps}_2^2}-1)\\
&\leq 2\norm{u-\vTildeEps}_2\;,
\end{align*}
since $\sqrt{1+x^2}\leq 1+x$ for $x\geq 0$.
Let us call $\mu/(pn^{\alpha_{min}})^{1/2}=\xi$ and $d=\lambda_1-\lambda_2$.
We show in Theorem \ref{thm:SecondOrderAccuracyRegularizedAveragingProc} that, for some $\kappa>0$, asymptotically
$$
\Expect\left[\norm{u-\vTildeEps}_2\right]\leq \kappa (\frac{\xi^2}{d^2}+\frac{\xi^3}{d^3\eps})\;
$$
so when $\eps>\xi/d$, we have $\Expect\left[\norm{u-\vTildeEps}_2\right]\leq \kappa \frac{\xi^2}{d^2}$ and therefore
$$
\Expect\left[\norm{u-\beta}_2\right]\leq \kappa \frac{\xi^2}{d^2}.
$$
\paragraph{Control of $\bm{P(\nu\neq \beta).}$} We have (essentially) seen in the proof of Theorem \ref{thm:ApproxEigenvectorsOrderk+2Accurate} above that if $2\opnorm{E}/d<1-\eps$, then $\opnorm{(\id+\Delta)^{-1}}\leq 1/\eps$ (see also the proof of Theorem \ref{thm:SecondOrderAccuracyRegularizedAveragingProc}). Hence
$$
P\left(\opnorm{(\id+\Delta)^{-1}}> 1/\eps\right)\leq P\left(\opnorm{E}>\frac{(1-\eps)d}{2}\right)\;.
$$
Recall that we have now chosen $\eps=2 \xi/d$. In that case, we have
$$
\frac{(1-\eps)d}{2}=\frac{d}{2}-\xi\;.
$$
Now we show the following deviation inequality in Theorem \ref{thm:ControlOfE}: if $m_E$ is a median of $\opnorm{E}$,
$$
P\left(\left|\opnorm{E}-m_E\right|>t\right)\leq 4 \exp\left(-\frac{p^2}{8\norm{M}_{\infty}^2}t^2\right)\;.
$$
Recall also that for $n$ large enough $0\leq m_E\leq 3 \xi$ when the conditions of Theorem \ref{Thm:ControlOpNormErrorMatrix} apply (see Theorems \ref{Thm:ControlOpNormErrorMatrix} or arguments at the end of the proof of Theorem \ref{thm:ControlMaxOpNormsMatrices}). Suppose now that $n$ is such
that indeed $m_E\leq 3 \xi$.
Then if $\frac{d}{2}-4\xi>0$, we have
$$
P\left(\opnorm{E}>\frac{(1-\eps)d}{2}\right)\leq P\left(\left|\opnorm{E}-m_E\right|>\frac{(1-\eps)d}{2}-m_E\right)\leq P\left(\left|\opnorm{E}-m_E\right|>\frac{d}{2}-4\xi\right)\;.
$$
Now when $\xi/d\tendsto 0$, $\frac{d}{2}-4\xi\geq \frac{d}{3}$ asymptotically. Since we assumed that $d\geq \xi\sqrt{\ln(\xi^{-2})}$ and $\xi\tendsto 0$, we indeed have $\xi/d\tendsto 0$.
Therefore,
$$
P\left(\opnorm{E}>\frac{(1-\eps)d}{2}\right)\leq 4 \exp\left(-\frac{p^2}{72\norm{M}_{\infty}^2}d^2\right).
$$
All we have to do now is to verify that the asymptotics we consider, the quantity on the right-hand side of the previous equation remains less than $\xi^2/d^2$ asymptotically. Elementary algebra shows that this is equivalent to saying that
\begin{equation}\label{eq:boundOnd}
d^2-72\frac{\norm{M}_{\infty}^2}{p^2}\ln(d^2)\geq 72\frac{\norm{M}_{\infty}^2}{p^2}(-\ln(\xi^2)+\ln 4)\;.
\end{equation}
We have $\norm{M}_{\infty}^2/p^2=\lo{\xi^2}$, so the right-hand side is going to zero. In particular, we see that when $d\geq \xi\sqrt{\ln(\xi^{-2})}$, as we assume, the inequality above is satisfied asymptotically. As a matter of fact, when $d<\exp(1)$,
$$
d^2-72\frac{\norm{M}_{\infty}^2}{p^2}\ln(d^2)\geq d^2\;,
$$
and the result comes out of the fact that $\frac{\norm{M}_{\infty}^2}{p^2}=\lo{\xi^2}$. If $d>\exp(1)$, the result is obvious as the right-hand side of Equation \eqref{eq:boundOnd} goes to 0 asymptotically, while the left-hand side
is asymptotically larger than $\exp(2)/2$ for instance.
So we have shown that under our assumptions,
$$
P(\nu\neq \beta)\leq \frac{\xi^2}{d^2}\;.
$$
We can finally conclude that
$$
\Expect[\opnorm{\nu-u}]\leq \kappa \frac{\xi^2}{d^2}\;,
$$
as announced in the theorem.
\end{proof}
This result applies to all eigenvectors corresponding to eigenvalues whose isolation distance (i.e distance to the nearest eigenvalue) satisfies the separation condition \eqref{eq:pert-condition-Strong}, which is a strong version of the separation condition \eqref{eq:pert-condition}. We note that we need the strong separation condition (Equation \eqref{eq:pert-condition-Strong}) to be able to take expectations rigorously.

Finally, we note that theoretical as well as practical considerations seem to indicate that condition (\ref{eq:pert-condition}) (and hence \eqref{eq:pert-condition-Strong}) is quite conservative. On the theoretical side, we see with Equation \eqref{eq:ExactKthOrderExpansionOfEigenvector} that what really matters for the quality of the approximation is the norm of the vector
$$
l_j=\Delta^{j+2}(\id+\Delta)^{-1}u\;,
$$
or its expectation. We used in our approximations the coarse bound $\opnorm{\Delta}\leq 2 \opnorm{R}\opnorm{E}$, which is convenient because it does not require us to have information about the eigenvectors of $\Delta$. However, we see that the norm of $l_j$ could be small even when $\opnorm{R}\opnorm{E}$ is not very small, for instance if $u$ belonged to a subspace spanned by eigenvectors of $\Delta$ associated with eigenvalues of this matrix that are small in absolute value. So it is quite possible that our method could work in a somewhat larger range of situations than the one for which we have theoretical guarantees. This is what our simulations below seem to indicate.

\subsection{Variance}
\label{ss:variance} The expansion in Equation $\eqref{eq:eigv-expansion}$ also allows us to approximate the variance of the first-order residual $REu$ after subsampling. This is useful in practice because it gives us an idea of how many independent computations we need to make to essentially void the effect of the first order term in the expansion of $v$. In terms of distributed computing, it therefore tells us how many machines we should involve in the computation. We have the following theorem.

\begin{theorem}\label{thm:VarianceREu}
Let $u_1$ be the eigenvector associated with $\lambda_1$, the largest eigenvalue of $M$. Let us call $w_1=u_1\circ u_1$, and ${\cal M}=M\circ M$. Then
\[
\Expect[\|REu_1\|_2^2] \leq \frac{1}{(\lambda_2-\lambda_1)^2} \frac{1-p}{p} \left(\sum_{k=1}^n u_1(k)^2 \|M_k\|_2^2-\left[2 w_1^T {\cal M} w_1 -\sum_{k=1}^n w_1^2(k) {\cal M}_{kk}\right]\right)\;.
\]
Assuming w.l.o.g. that $\lambda_1=\|M\|_2$, this bound yields in particular
\BEQ\label{eq:var-eigv}
\Expect[\|REu_1\|_2^2] \leq \frac{1}{(1-\lambda_2/\lambda_1)^2} \|u_1\|_{\infty}^2 \frac{\NumRank(M)}{p}
\EEQ
where $\NumRank(M)=\|M\|_F^2/\|M\|_2^2$ is the numerical rank of the matrix $M$ and is a stable relaxation of the rank, satisfying $1\leq \NumRank(M) \leq \Rank(M) \leq n$ (see \cite{Rude07} for a discussion).
\end{theorem}

\begin{proof}
By construction, $\Expect[E]=0$ and
\BEAS
\Expect[\|REu_1\|_2^2] &=& \Expect[u_1^TER^2Eu_1]= \sum_{j=2}^n \Expect[\frac{(u_1^TEu_j)^2}{(\lambda_j-\lambda_1)^2}],
\EEAS
by definition of $R$. Now
$$
\sum_{j=1}^n (u_1^TEu_j)^2 = \norm{Eu_1}_2^2=u_1^TE^2 u_1\;,$$
because $E$ is symmetric, the $u_i$'s form an orthonormal basis and $u_1^TEu_j$ is the $j$-th coefficient of $Eu_1$ in this basis, so the sum of the squared coefficients is the squared norm of the vector. Hence
\BEAS
\Expect[\|REu_1\|_2^2]\leq \frac{1}{(\lambda_2-\lambda_1)^2}\left(\Expect[u_1^TE^2 u_1]-\var(u_1^T E u_1)\right)\;.
\EEAS
The variance of $u_1^T E u_1$ is easy to compute if we rewrite this quantity as a sum of independent random variables.
Also, separate computations (see Appendix, Subsection \ref{ss:AppendixVarianceComputations}) show that $\Expect[E^2]$ is a diagonal matrix, whose $i$-th diagonal entry is $(1-p)\|M_i\|_2^2/p$, where $M_i$ is the $i$-th column of $M$.  Hence, in that case, having defined $w_1=u_1\circ u_1$ and ${\cal M}=M\circ M$, we get
\[
\Expect[\|REu_1\|_2^2] \leq \frac{1}{(\lambda_2-\lambda_1)^2} \frac{1-p}{p} \left(\sum_{k=1}^n u_1(k)^2 \|M_k\|_2^2-\left[2 w_1^T {\cal M} w_1 -\sum_{k=1}^n w_1^2(k) {\cal M}_{kk}\right]\right)\;.
\]
Assuming w.l.o.g. that $\lambda_1=\|M\|_2$, we get (\ref{eq:var-eigv}).
\end{proof}

\subsection{Nonsymmetric matrices}
\label{ss:nonsym} The results described above are easily extended to nonsymmetric matrices. Here  $M\in\reals^{m\times n}$, with $m\geq n$ and we write its spectral decomposition
\[
M=\sum_{i=1}^n \sigma_i u_iv_i^T,
\]
where $u_i\in\reals^n$, $v_i\in \reals^m$ and $\sigma_i>0$. We can adapt the definition of incoherence to
\[
\mu(M,\alpha,\beta)=\sum_{i=1}^n \sigma_i n^{\alpha_i/2}\|u_i\|_\infty m^{\beta_i/2}\|v_i\|_\infty
\]
and reformulate our main assumption on $M$ as follows.
\begin{assumption}
There are vectors $\alpha\in[0,1]^n$ and $\beta\in[0,1]^n$ for which
\[
\mu(M,\alpha,\beta)\leq \mu \quad \mbox{and} \quad \Card(u_i)\leq n^{\alpha_i},~\Card(v_i)\leq m^{\beta_i}, \quad i=1,\ldots,n
\]
as $m,n$ go to infinity with $m=\rho n$ for a given $\rho>1$, where $\mu$ is an absolute constant.
\end{assumption}
In this setting, using again \cite[Th.~5.5.19]{Horn91}, we get
\BEQ
\left\| \sum_{i=1}^n \sigma_i C\circ  (u_iv_i^T)  \right\|_2 \leq \sum_{i=1}^n \sigma_i n^{\alpha_i/4}\|u_i\|_\infty m^{\beta_i/4}\|v_i\|_\infty \left\| \frac{C_{\alpha_i,\beta_i}}{n^{\alpha_{i/4}}m^{\beta_i/4}} \right\|_2
\EEQ
where we have assumed that $u_i,v_i$ are sparse and $C_{\alpha_i,\beta_i}$ is a $n^{\alpha_i}\times m^{\beta_i}$ submatrix of~$C$. As in (\ref{eq:norm-bound}), we can then bound the spectral norm of the residual and we have
\BEQ\label{eq:norm-bound-nonsym}
\limsup_{n\tendsto \infty} \|E\|_2 \leq \frac{2 \mu}{\sqrt{pn^{\frac{\alpha_{\mathrm{min}}}{2}}m^{\frac{\beta_{\mathrm{min}}}{2}}}}.
\EEQ
almost surely. Perturbation results similar to (\ref{eq:eigv-expansion}) for left and right eigenvectors are detailed in \cite{Stew98} for example.
\section{Numerical experiments}
\label{s:mum}
In this section, we study the numerical performance of the subsampling/averaging results detailed above on both artificial and realistic data matrices

\paragraph{Dense matrices: PCA, SVD, etc.}
We first illustrate our results by approximating the leading eigenvector of a matrix $M$ as the average of leading eigenvectors of subsampled matrices, for various values of the sampling probability $p$. To start with a naturally structured dense matrix, we form $M$ as the covariance matrix of the 500 most active genes in the colon cancer data set in \cite{Alon99}. We let $p$ vary from $10^{-4}$ to 1 and for each $p$, we compute the leading eigenvector of 1000 subsampled matrices, average these vectors and normalize the result. We call $u$ the true leading eigenvector of $M$ and $v$ the approximate one. We now normalize $v$ so that $\norm{v}_2=1$ (which is standard, but different from the normalization we used in our theoretical investigations where we had $u\trsp v=1$).

In Figure \ref{fig:phase-transition}, we plot $u^Tv$ as a function of $p$ together with the median of $u^Tv$ sampled over all individual subsampled matrices, with dotted lines at plus and minus one standard deviation. We also record the proportion of samples where $\|E\|$ satisfies the perturbation condition (\ref{eq:pert-condition}).

\begin{figure}[ht]
\begin{center}
\begin{tabular}{cc}
\psfrag{p}[t][b]{$p$}
\psfrag{cos}[b][t]{$u^Tv$}
\includegraphics[width=0.49 \textwidth]{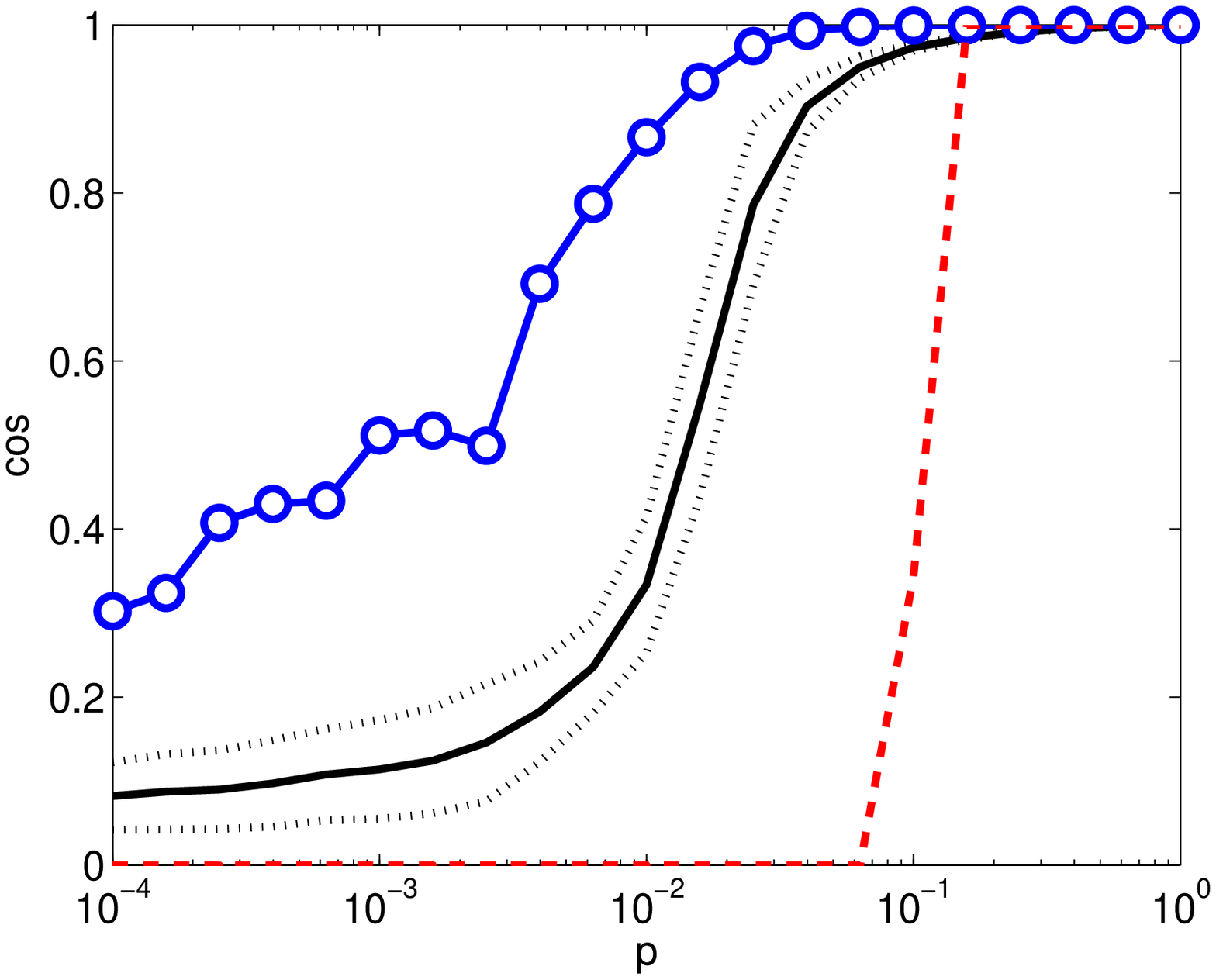}&
\psfrag{p}[t][b]{$p$}
\psfrag{cos}[b][t]{$u^Tv$}
\includegraphics[width=0.49\textwidth]{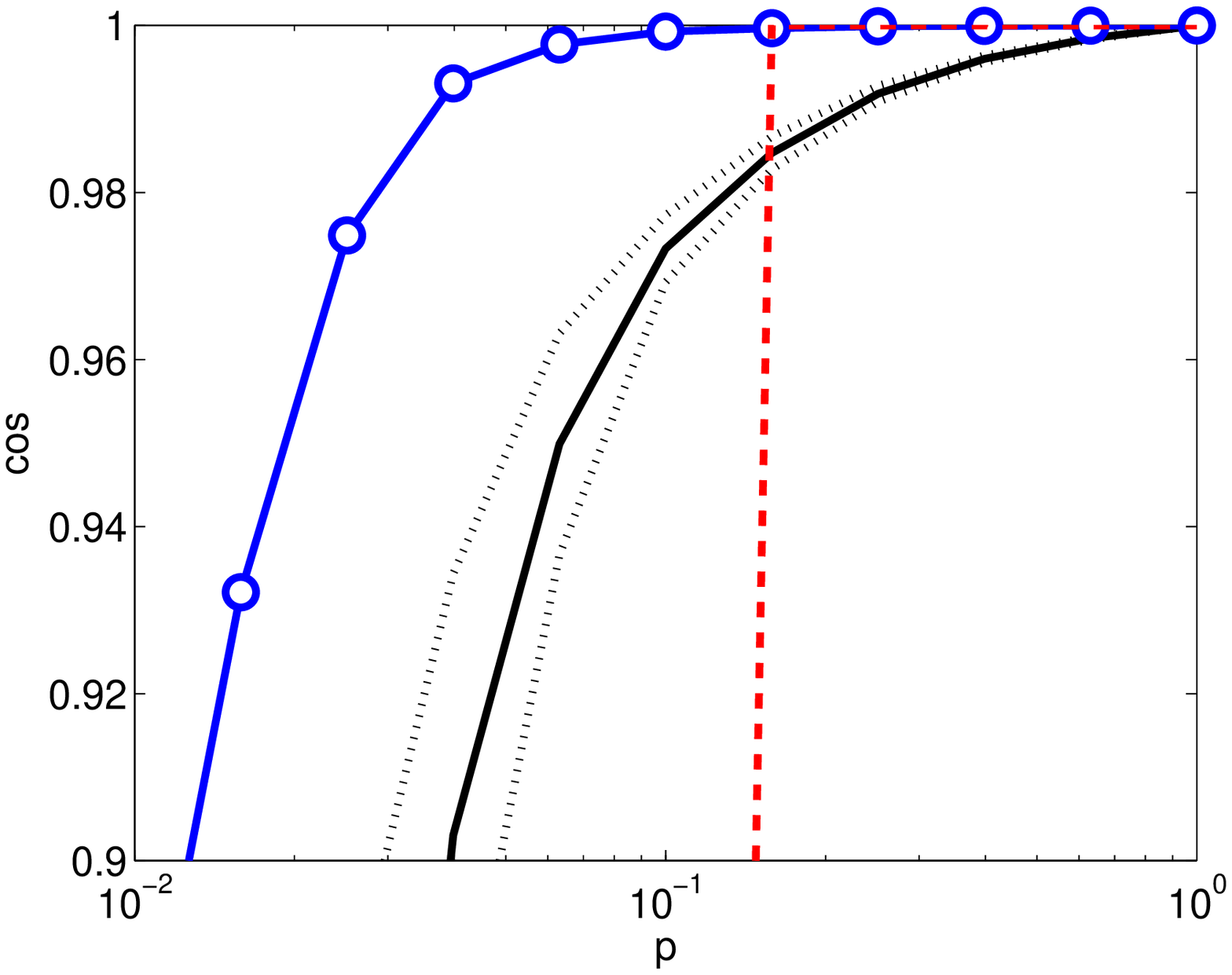}
\end{tabular}
\caption{\textit{Left:} Alignment $u^Tv$ between the true and the normalized average of 1000 subsampled eigenvectors (blue circles), median value of $u^Tv$ over all sampled matrices (solid black line), with dotted lines at plus and minus one standard deviation and proportion of samples satisfying the perturbation condition (\ref{eq:pert-condition}) (dashed red line), for various values of the sampling probability $p$ on a gene expression covariance matrix. \textit{Right:} Zoom on the the interval $p\in[10^{-2},1]$.
\label{fig:phase-transition}}
\end{center}
\end{figure}

We repeat this experiment on a (nonsymmetric) term-document matrix formed using press release data from PRnewswire, to test the impact of subsampling on Latent Semantic Indexing results. Once again, we let $p$ vary from $10^{-2}$ to 1 and for each $p$, we compute the leading eigenvector of 1000 subsampled matrices, average these vectors and normalize the result. We call $u$ the true leading eigenvector of $M$ and $v$ the approximate one. In Figure \ref{fig:phase-transition-svd} on the left, we plot $u^Tv$ as a function of $p$ together with the median of $u^Tv$ sampled over all individual subsampled matrices, with dotted lines at plus and minus one standard deviation. The matrix $M$ is $6779\times11171$ with spectral gap $\sigma_2/\sigma_1=0.66$. 

In Figure \ref{fig:phase-transition-svd} on the right, we plot the ratio of CPU time for subsampling a gene expression matrix of dimension 2000 and computing the leading eigenvector of the subsampled matrix (on a single machine), over CPU time for computing the leading eigenvector of the original matrix. Two regimes appear, one where the eigenvalue computation dominates with computation cost scaling with $p$, another where the sampling cost dominates and the speedup is simply the ratio between sampling time and the CPU cost of a full eigenvector computation. Of course, the principal computational benefit of subsampling is the fact that memory usage is directly proportional to $p$.

\begin{figure}[ht]
\begin{center}
\begin{tabular}{cc}
\psfrag{p}[t][b]{$p$}
\psfrag{cos}[b][t]{$u^Tv$}
\includegraphics[width=0.49 \textwidth]{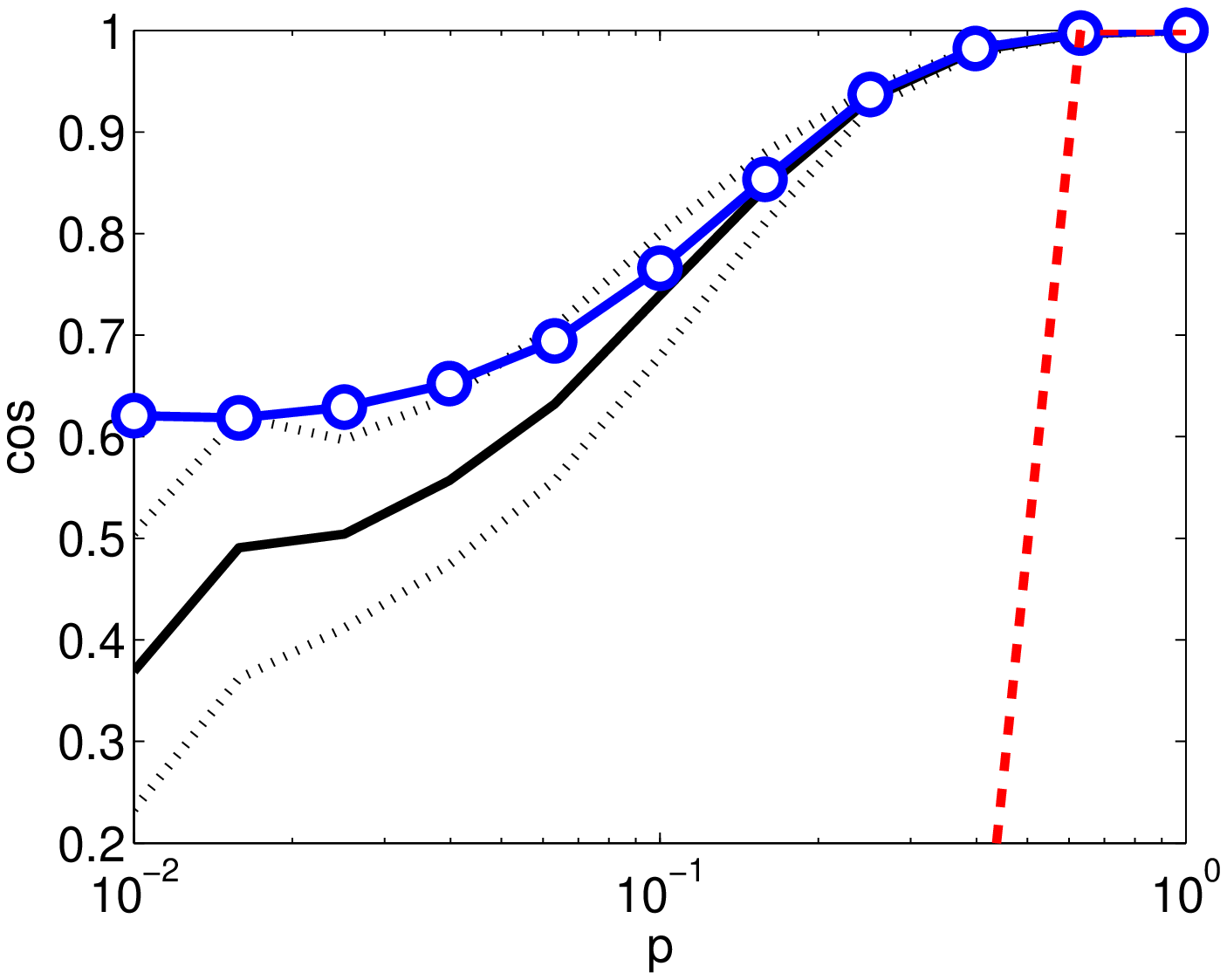}&
\psfrag{p}[t][b]{$p$}
\psfrag{speedup}[b][t]{Speedup}
\includegraphics[width=0.49\textwidth]{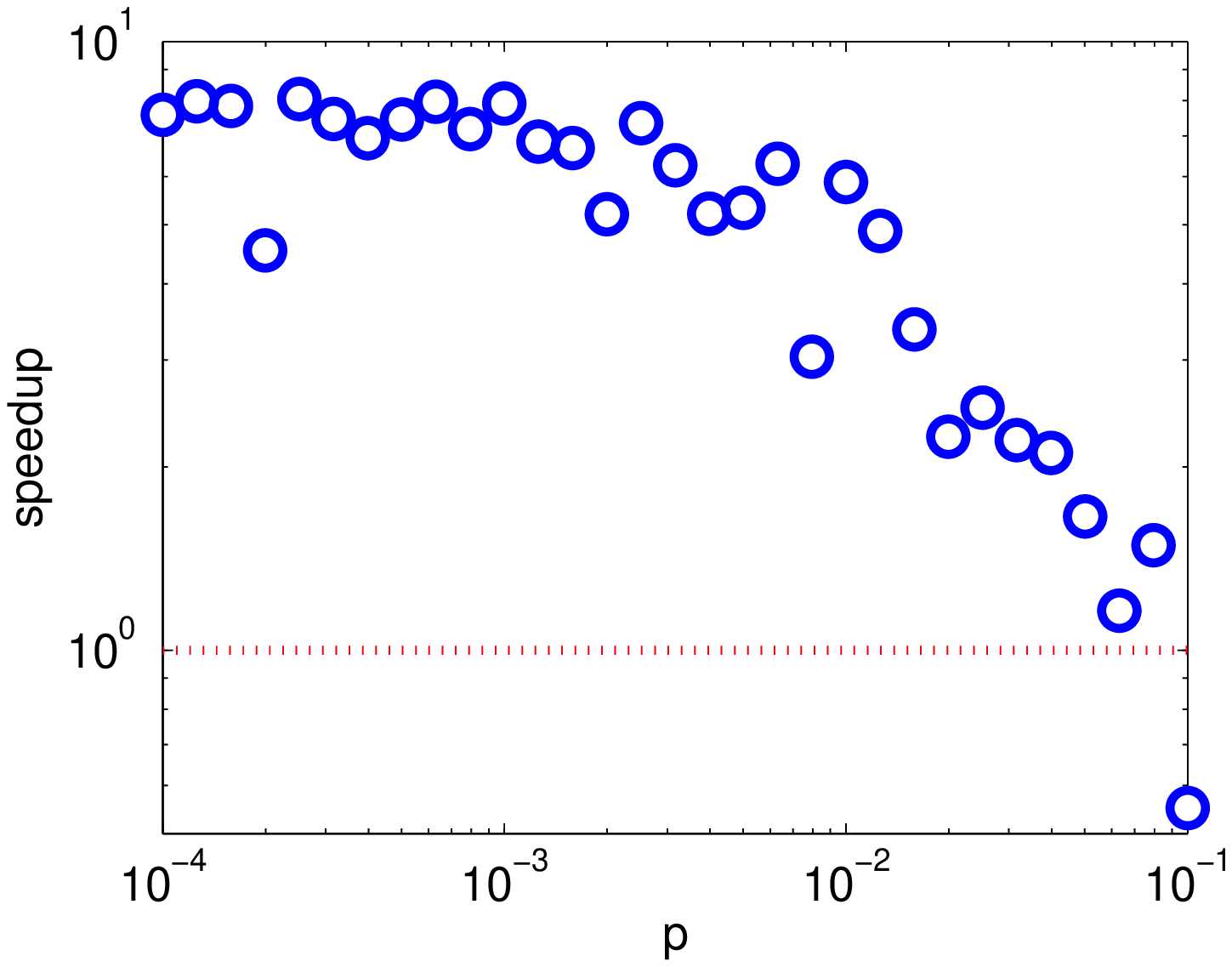}
\end{tabular}
\caption{\textit{Left:} Alignment $u^Tv$ between the true and the normalized average of 1000 subsampled left eigenvectors (blue circles), median value (solid black line), dotted lines at plus and minus one standard deviation and proportion of samples satisfying condition (\ref{eq:pert-condition}) (dashed red line), for various values of the sampling probability $p$ on a term document matrix with dimensions $6779\times11171$. \textit{Right:} Speedup in computing leading eigenvectors on gene expression data, for various values of the sampling probability $p$.
\label{fig:phase-transition-svd}}
\end{center}
\end{figure}

A key difference between the experiments of Figure \ref{fig:phase-transition} and those of \ref{fig:phase-transition-svd} is that the leading eigenvector of the gene expression data set is much more incoherent than the leading left eigenvector of the term-document matrix, which explains part of the difference in performance. We compare both eigenvectors in Figure~\ref{fig:eigenvectors}.

\begin{figure}[ht]
\begin{center}
\begin{tabular}{cc}
\psfrag{ii}[t][b]{$i$}
\psfrag{uu}[b][t]{$|u_i|$}
\includegraphics[width=0.49 \textwidth]{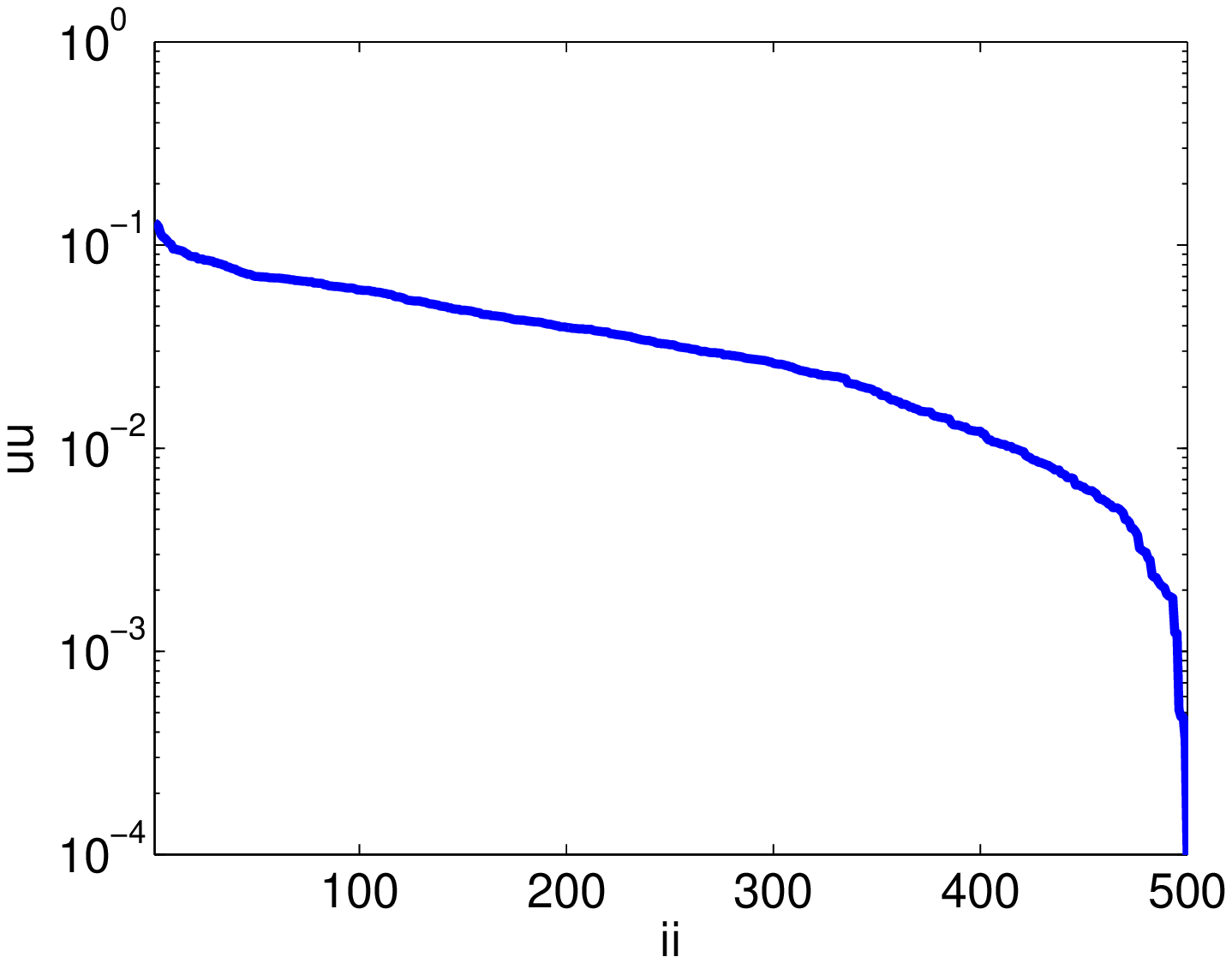}&
\psfrag{ii}[t][b]{$i$}
\psfrag{uu}[b][t]{$|u_i|$}
\includegraphics[width=0.49\textwidth]{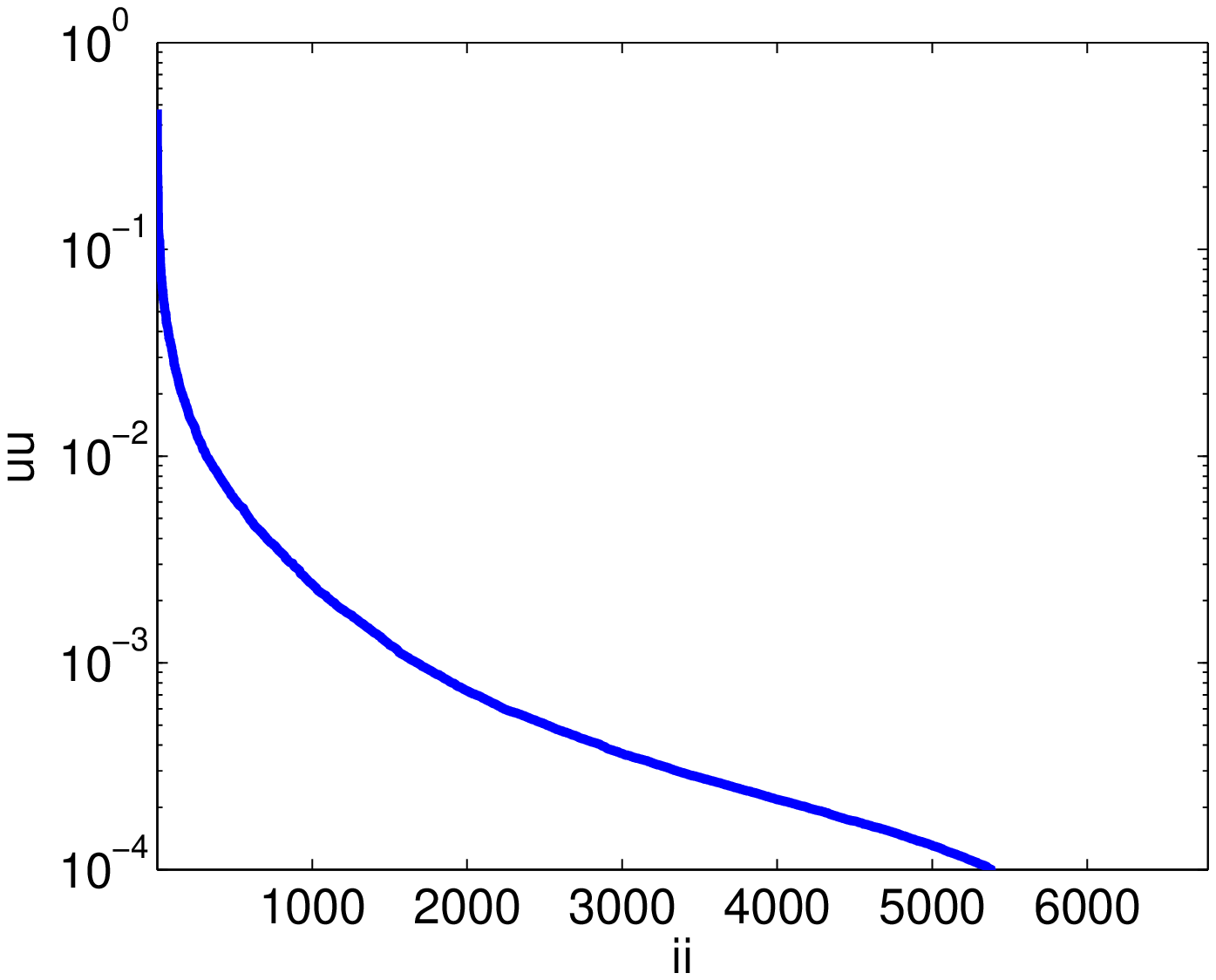}
\end{tabular}
\caption{Magnitude of eigenvector coefficients $|u_i|$ in decreasing order for both the leading eigenvector of the gene expression covariance matrix (left) and leading left eigenvector of the $6779\times11171$ term document matrix (right).
\label{fig:eigenvectors}}
\end{center}
\end{figure}

We then study the impact of the number of samples on precision. We use again the colon cancer data set in \cite{Alon99}. In Figure~\ref{fig:sensib} on the left, we fix the sampling rate at $p=10^{-2}$ and plot $u^Tv$ as a function of the number of samples used in averaging. We also measure the impact of the eigenvalue gap $\lambda_2/\lambda_1$ on precision. We scale the spectrum of the gene expression covariance matrix so that its first eigenvalue is $\lambda_1=1$ and plot the alignment $u^Tv$ between the true and the normalized average of 100 subsampled eigenvectors over subsampling probabilities $p\in[10^{-2},1]$ for various values of the spectral gap $\lambda_2/\lambda_1\in\{0.75,0.95,0.99\}$.

\begin{figure}[ht]
\begin{center}
\begin{tabular}{cc}
\psfrag{p}[t][b]{$p$}
\psfrag{cos}[b][t]{$u^Tv$}
\includegraphics[width=0.49 \textwidth]{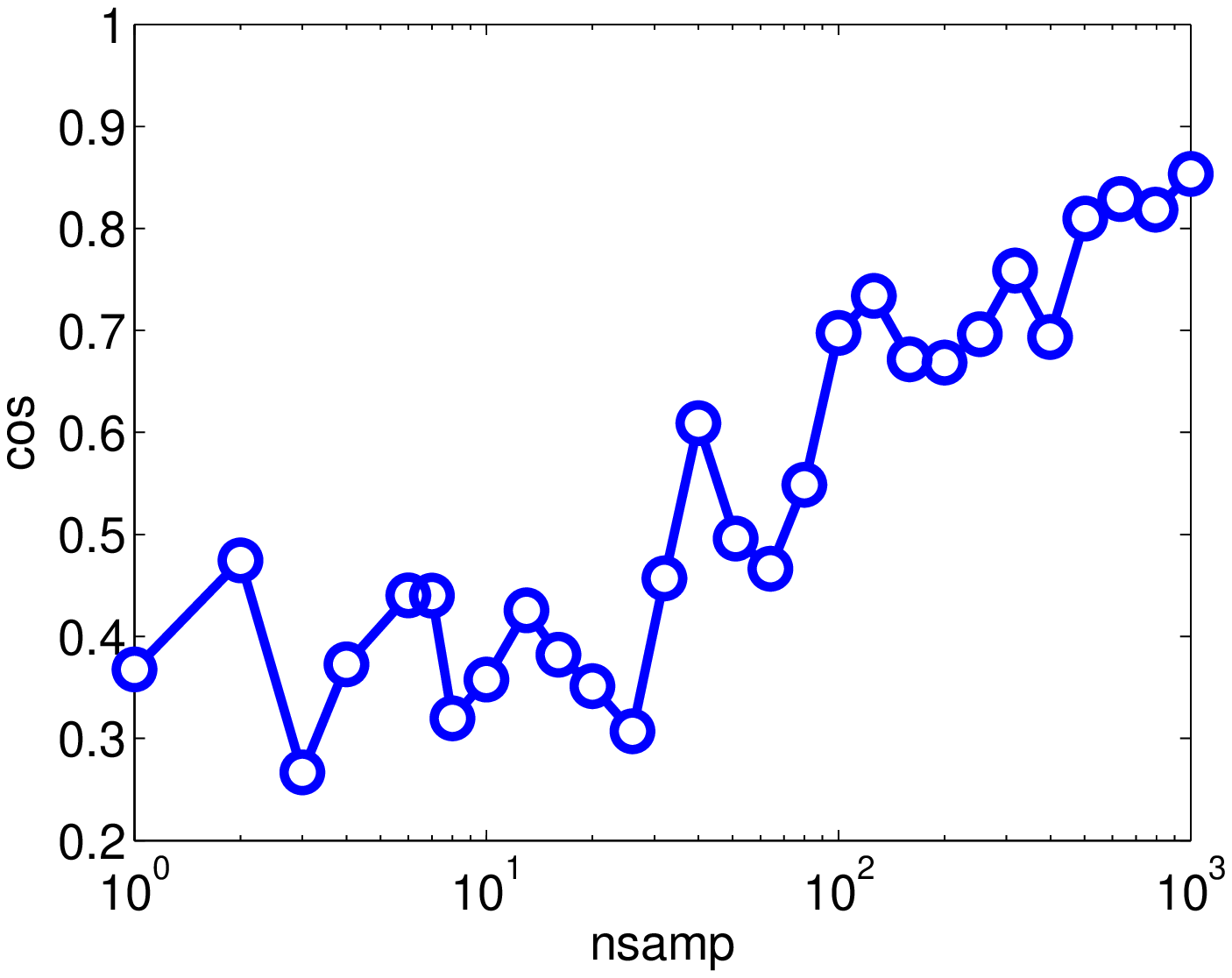}&
\psfrag{p}[t][b]{$p$}
\psfrag{cos}[b][t]{$u^Tv$}
\includegraphics[width=0.49\textwidth]{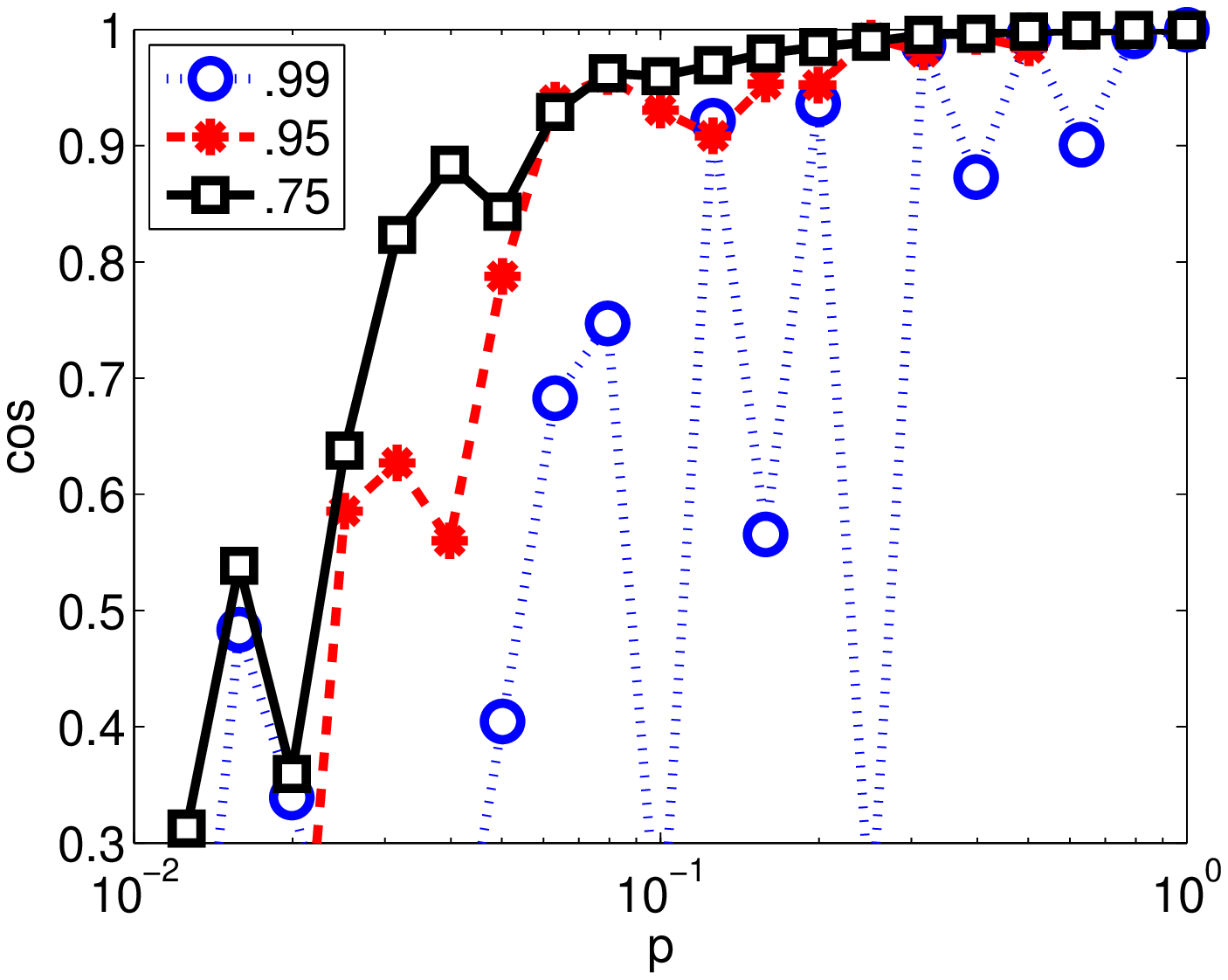}
\end{tabular}
\caption{\textit{Left:} Alignment $u^Tv$ between the true leading eigenvector $u$ and the normalized average leading eigenvector versus number of samples, on the gene expression covariance matrix with subsampling probability $p=10^{-2}$. \textit{Right:} Alignment~$u^Tv$ for various values of the spectral gap $\lambda_2/\lambda_1\in\{0.75,0.95,0.99\}$.
\label{fig:sensib}}
\end{center}
\end{figure}

\begin{figure}[ht]
\begin{center}
\begin{tabular}{cc}
\includegraphics[width=0.50 \textwidth]{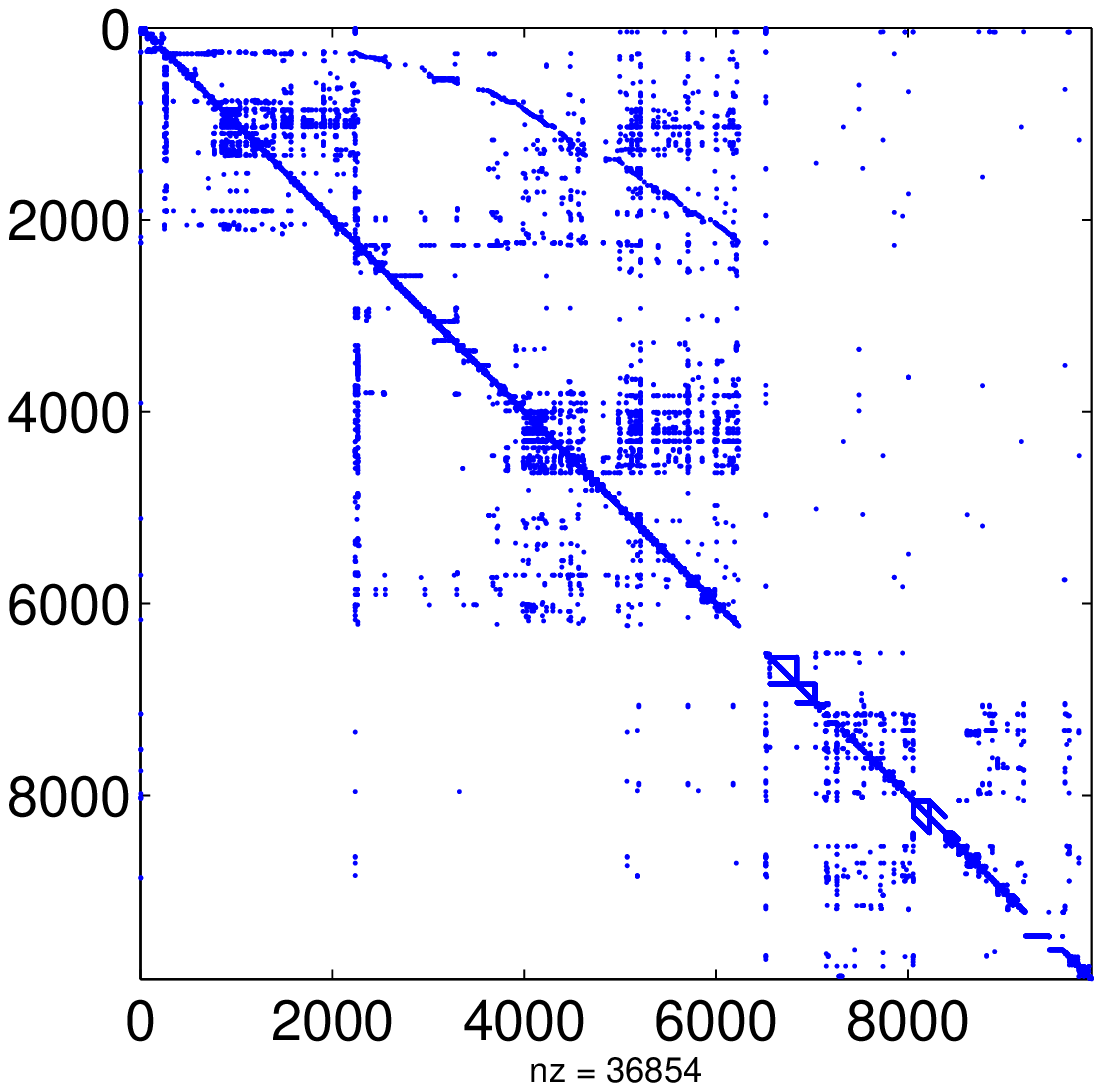}&
\psfrag{ii}[t][b]{$i$}
\psfrag{uu}[b][t]{$|u_i|$}
\includegraphics[width=0.49\textwidth]{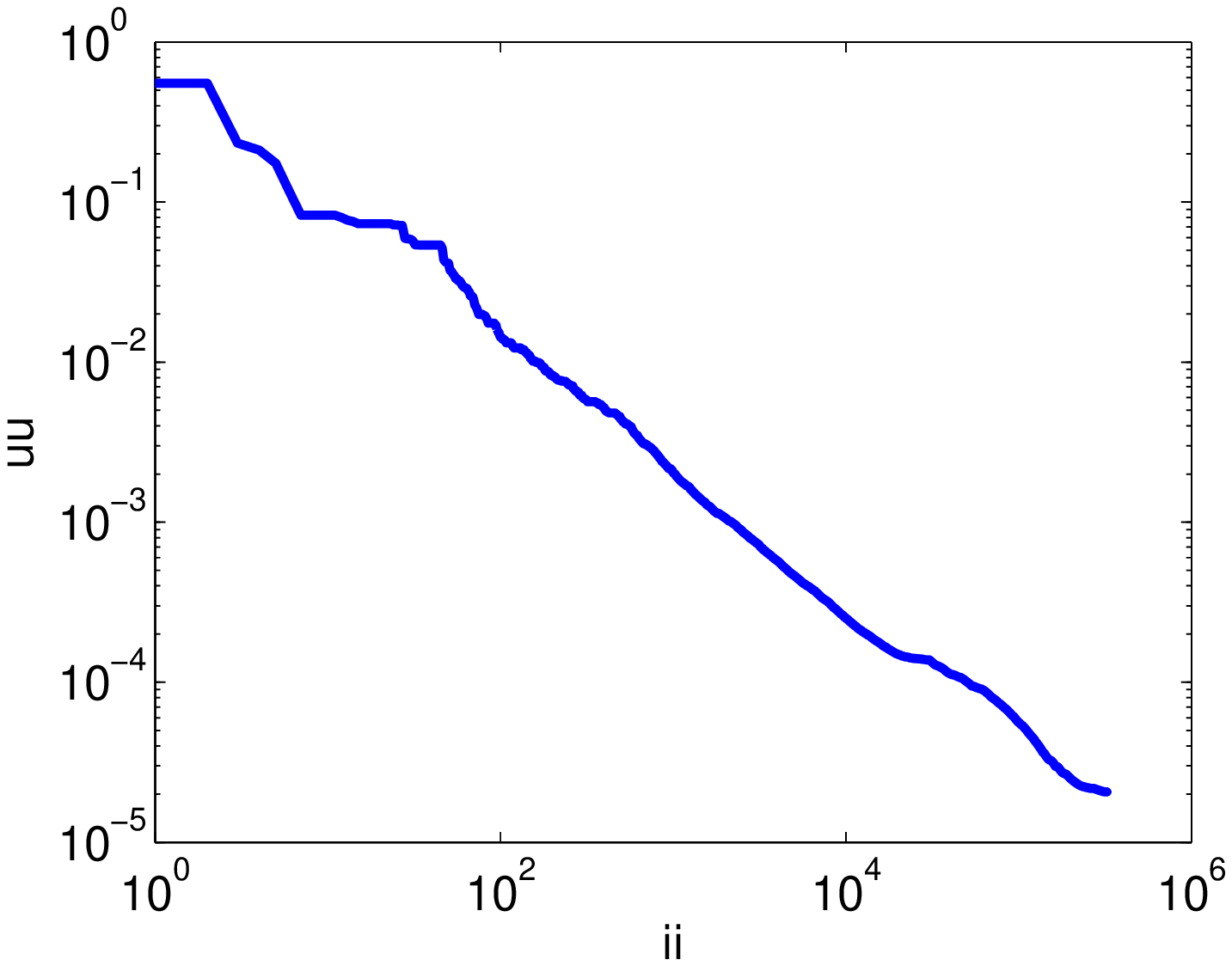}
\end{tabular}
\caption{\textit{Left:} The {\tt wb-cs.stanford} graph. \textit{Right:} Loglog plot of the Pagerank vector coefficients for the {\tt cnr-2000} graph.
\label{fig:web}}
\end{center}
\end{figure}

\begin{figure}[ht]
\begin{center}
\begin{tabular}{cc}
\psfrag{p}[t][b]{$p$}
\psfrag{cos}[b][t]{Spearman's $\rho$}
\includegraphics[width=0.49 \textwidth]{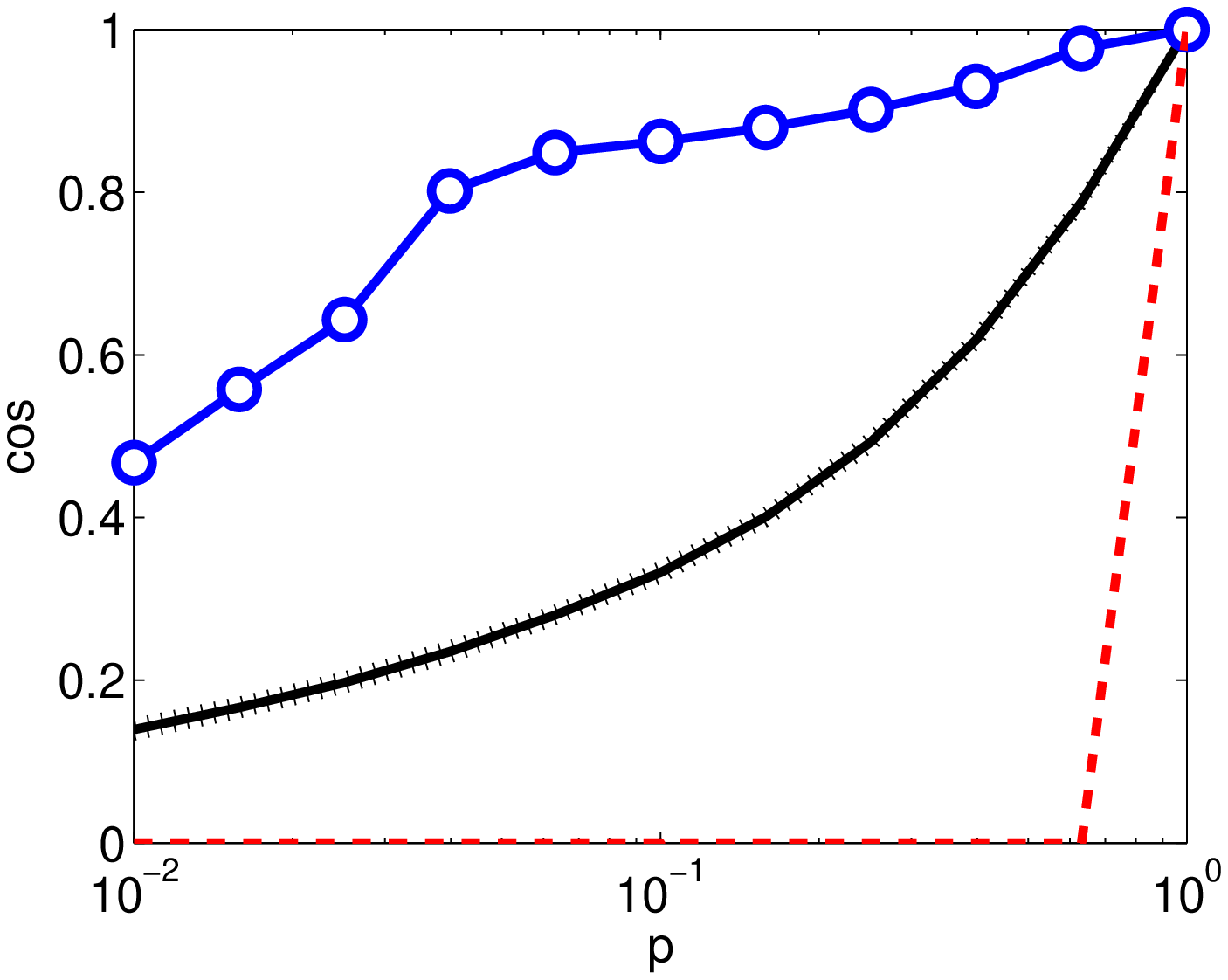}&
\psfrag{p}[t][b]{$p$}
\psfrag{cos}[b][t]{Spearman's $\rho$}
\includegraphics[width=0.49\textwidth]{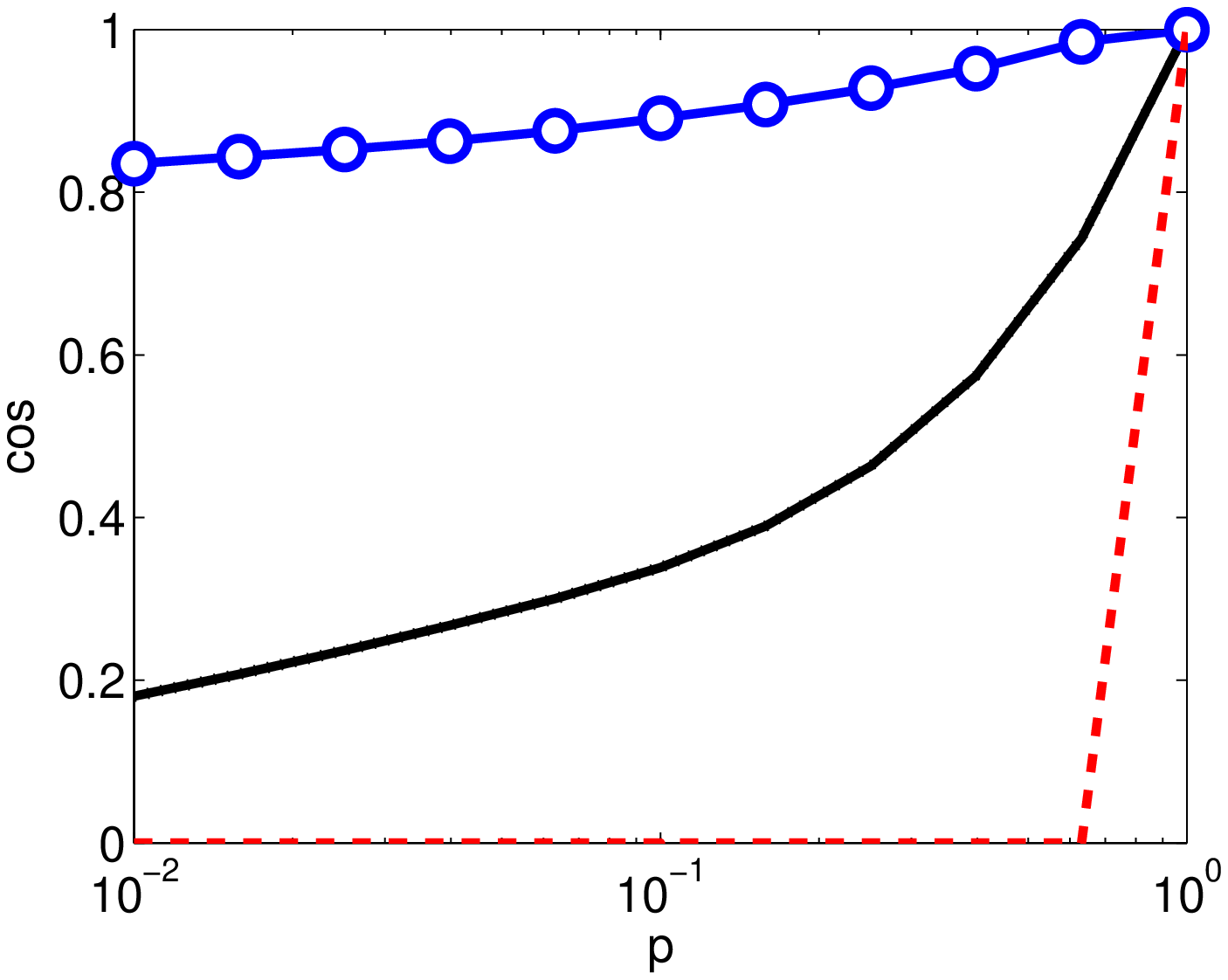}
\end{tabular}
\caption{Ranking correlation (Spearman's $\rho$) between true and averaged pagerank vector (blue circles), median value of the correlation over all subsampled matrices (solid black line), dotted lines at plus and minus one standard deviation and proportion of samples satisfying the perturbation condition (\ref{eq:pert-condition}) (dashed red line), for various values of the sampling probability $p$. \textit{Left:} On the {\tt wb-cs.stanford} graph. \textit{Right:} On the {\tt cnr-2000} graph.
\label{fig:rank-corr}}
\end{center}
\end{figure}

\paragraph{Graph matrices: ranking.}
Here, we test the performance of the methods described above on graph matrices used in ranking algorithms such as pagerank \cite{Page98} (because of its susceptibility to manipulations however, this is only one of many features used by search engines). Suppose we are given the adjacency matrix of a web graph, with
\[\left\{\BA{l}
A_{ij}=1, \quad \mbox{if there is a link from $i$ to $j$}\\
A_{ij}=0, \quad \mbox{otherwise},
\EA\right.\]
where $A\in\reals^{n \times n}$ (one such matrix is displayed in Figure~\ref{fig:web}). Whenever a node has no out-links, we link it with every other node in the graph, so that $B=A+\delta \ones^T/n$, with $\delta_i=1$ if and only if $\mathrm{deg}_i=0$, where $\mathrm{deg}_i$ is the degree of node $i$. We then normalize $B$ into a stochastic matrix $P^g_{ij}={B_{ij}}/{\mathrm{deg}_i}$. The matrix $P^g$ is the transition matrix of a Markov chain on the graph modeling the behavior of a web surfer randomly clicking on links at every page. For most web graphs, this Markov chain is usually not irreducible but if we set
\[
P=cP^g + (1-c) \ones\ones^T/n
\]
for some $c\in(0,1]$, the Markov chain with transition matrix $P$ will be irreducible. An additional benefit of this modification is that the spectral gap of $P$ is at least $c$ \cite{Have03}. The leading (Perron-Frobenius) eigenvector $u$ of this matrix is called the {\em Pagerank} vector \cite{Page98}, its coefficients $u_i$ measure the stationary probability of page $i$ being visited by a random surfer driven by the transition matrix~$P$, hence reflect the importance of page $i$ according to this model.

The coefficients of pagerank vectors typically follow a power law for classic values of the damping factor \cite{Pand06,Becc06} which means that the bounds in assumption 1 do not hold. Empirically however, while the distance between true and averaged eigenvectors quickly gets large, the ranking correlation (measured using Spearman's $\rho$ \cite{Melu07}) is surprisingly robust to subsampling.

We use two graphs from the Webgraph database \cite{Bold04}, {\tt wb-cs.stanford} which has 9914 nodes and 36854 edges, and {\tt cnr-2000} which has 325,557 nodes and 3,216,152 edges. For each graph, we form the transition matrix $P$ as in \cite{Glei04} with uniform teleportation probability and set the teleportation coefficient $c=0.85$. In Figure~\ref{fig:web} we plot the {\tt wb-cs.stanford} graph and the Pagerank vector for {\tt cnr-2000} in loglog scale. In Figure~\ref{fig:rank-corr} we plot the ranking correlation (Spearman's~$\rho$) between true and averaged Pagerank vector (over 1000 samples), the median value of the correlation over all subsampled matrices and the proportion of samples satisfying the perturbation condition~(\ref{eq:pert-condition}), for various values of the sampling probability $p$. We notice that averaging very significantly improves ranking correlation, far outside the perturbation regime.

\section{Conclusion}
We have proposed a method to compute the eigenvectors of very large matrices in a distributed fashion:
\begin{enumerate} \itemsep 0ex
\item To each node in a computer cluster of size $N$, we send a subsampled version $S_i$ of the matrix of interest, $M$.
\item Node $i$ computes the relevant eigenvectors of $S_i$.
\item The $N$ eigenvectors are averaged together and normalized to produce our final estimator.
\end{enumerate}
The key to the algorithm is that Step 2 is numerically cheap (because $S_i$ is very sparse), and hence can be executed fast even on small machines. Therefore a cluster or cloud of small machines could be used to approximate the eigenvectors of $M$, a difficult problem in general when $M$ is extremely large.

We have shown that under carefully stated conditions, the algorithm described above will yield a second-order accurate approximation of the eigenvectors of $M$. This gain in accuracy comes from the averaging step of our algorithm. We note that arguments similar to the ones we used in this paper could be made to compute second-order accurate approximations of the eigenvalues of $M$. (We restricted ourselves to eigenvectors here because in methods such as PCA, the eigenvectors are in some sense more important than the eigenvalues.) Our results depend on a measure of incoherence for $M$ that we propose in this paper. They also show that subsampling will work if the sampling probability is small, but is likely to fail if that probability is too small.

Finally, our simulations show that we gain significantly in accuracy by averaging subsampled eigenvectors (which suggests that our theoretical passage from first-order to second-order accuracy is also relevant in practice) and that the performance of our method seems to degrade for very incoherent matrices, a result that is also in line with our theoretical predictions.

\appendix
\vspace{1cm}
\renewcommand{\thesection}{\Alph{section}}
\renewcommand{\theequation}{\Alph{section}-\arabic{equation}}
\renewcommand{\thetheorem}{\Alph{section}-\arabic{theorem}}
\renewcommand{\thesubsection}{\Alph{section}-\arabic{subsection}}
\renewcommand{\thesubsubsection}{\Alph{section}-\arabic{subsection}.\arabic{subsubsection}}
\setcounter{equation}{0}  
\setcounter{theorem}{0}
\setcounter{section}{0}

\section{Appendix}
\subsection{On $\bm{\|C\|_2}$}\label{ss:AppControlNormCAndE}
Let us consider the symmetric random matrix $C$ with entries distributed as, for $i\geq j$,
\begin{equation}\label{eq:DistEntriesMatrixC}
C_{i,j}=\begin{cases}
\sqrt{\frac{1-p}{p}} &\text{ with probability } p\\
-\sqrt{\frac{p}{1-p}} &\text{ with probability } 1-p
\end{cases}\;.
\end{equation}
We assume that $C$ is $n\times n$. Our aim is to show that we can control $\opnorm{C}$ and in particular its deviation around its median. We do so by using Talagrand's inequality.

We have the following theorem.

\begin{theorem}\label{thm:ControlMaxOpNormsMatrices}

Suppose that we observe $n$ matrices $C_{\alpha_i}$, for $1\leq i\leq n$ with entries distributed as those of the matrix $C$ just described. Suppose these matrices are of size $n^{\alpha_i}$, where $\alpha_i$ are positive numbers. Call $\alpha_{\min}=\min_{1\leq i \leq n} \alpha_i$ and assume that, for some fixed $\delta>0$, $\alpha_{\min}>(\log n)^{(\delta-3)/4}$. Suppose further that $p$ is such that $\lim_{n\tendsto \infty} (\alpha_{\min} \log n)^4/(n^{\alpha_{\min}}p)=0$. Then
\begin{equation}\label{eq:ControlOpNormManyMatrices}
\limsup_{n\tendsto \infty}\opnorm{\frac{C_{\alpha_i}}{n^{\alpha_i/2}}} \leq 2 \text{ a.s }\;.
\end{equation}

\end{theorem}
\begin{proof}
We note that the application $C\tendsto \|C\|_2$ is a convex, $\sqrt{2}$-Lipschitz (with respect to Euclidian/Frobenius norm) function of the entries of $C$ that are on or above the main diagonal. As a matter of fact, since $\|\cdot\|$ is a norm, it is convex. Furthermore, if $A$ and $B$ are two symmetric matrices,
$$
\opnorm{A-B} \leq \Frobnorm{A-B}=\sqrt{\sum_{i,j}(a_{i,j}-b_{i,j})^2}\leq \sqrt{2}\sqrt{\sum_{i\leq j}(a_{i,j}-b_{i,j})^2}
$$
Now recall the consequence of Talagrand's inequality \cite{Talagrand95ConcProd} spelled out in \cite{ledoux2001}, Corollary 4.10 and Equation (4.10): if $F$ is a convex, $1$-Lipschitz function (with respect to Euclidian norm) on $\mathbb{R}^n$, of $n$ independent random variables ($X_1,\ldots,X_n$) that take value in $[u,v]$, and if $m_F$ is a median of $F(X_1,\ldots,X_n)$, then
\begin{equation}\label{eq:TalagrandInequality}
P(|F-m_F|>t)\leq 4\exp(-t^2/[4(u-v)^2])\;.
\end{equation}
The random variables that are above the main diagonal of $C$ are bounded, and take value in $[-\sqrt{\frac{p}{1-p}},\sqrt{\frac{1-p}{p}}]$. We note that
$$
\left(\sqrt{\frac{1-p}{p}}+\sqrt{\frac{p}{1-p}}\right)^2=\frac{1}{p(1-p)}\;.
$$
Therefore, calling $m_{n}$ the median of $\opnorm{n^{-1/2}C}$, we have, in light of Equation \eqref{eq:TalagrandInequality},
\begin{equation}\label{eq:ControlDeviationFromMedianOpnorm}
P\left(\left|\opnorm{\frac{C}{n^{1/2}}}-m_{n}\right|>t\right)\leq 4\exp\left(-\frac{nt^2}{8/(p(1-p))}\right)=4\exp\left(-\frac{t^2}{8}p(1-p)n\right)\;.
\end{equation}


Suppose now that we have a collection $C_{\alpha_i}$ of matrices of size $n^{\alpha_i}$ with entries distributed as in Equation \eqref{eq:DistEntriesMatrixC}. (Note that the matrices could be dependent.) Let us call $m_{n^{\alpha_i}}$ the medians of $\opnorm{C_{\alpha_i}/n^{\alpha_i/2}}$. Then we have, by a simple union bound argument, for any $k$,
$$
P\left(\max_{1\leq i \leq k}  \left|\opnorm{\frac{C_{\alpha_i}}{n^{\alpha_i/2}}}-m_{n^{\alpha_i}}\right|>t\right)\leq 4\sum_{i=1}^k \exp\left(-\frac{t^2}{8}p(1-p)n^{\alpha_i}\right)\leq 4k \exp\left(-\frac{t^2}{8}p(1-p)n^{\alpha_{\min}}\right)\;,
$$
where $\alpha_{\min}=\min_{1\leq i \leq k} \alpha_i$.

Suppose now that $k=n$, $p\leq 1/2$, $pn^{\alpha_{\min}}>(\log n)^{1+\delta}$, and $t\geq (\log n)^{-\delta/3}$ for some $\delta>0$. Then, $t^2 p(1-p)n^{\alpha_{\min}}>(\log n)^{1+\delta/3}/2$, which tends to $\infty$ as $n\tendsto \infty$. Because $u_n=n\exp(-(\log n)^{1+\delta/3}/16)$ is the general term of a converging series, we have, when $p\leq 1/2$ and $pn^{\alpha_{\min}}>(\log n)^{1+\delta}$ for some $\delta>0$,
$$
\max_{1\leq i \leq n}  \left|\opnorm{\frac{C_{\alpha_i}}{n^{\alpha_i/2}}}-m_{n^{\alpha_i}}\right|<(\log n)^{-\delta/3}\text{ a.s }\;,
$$
by a simple application of the Borel-Cantelli lemma. Hence, we have
\begin{equation}\label{eq:ControlOpNormManyMatricesThroughMedians}
\max_{1\leq i \leq n} \opnorm{\frac{C_{\alpha_i}}{n^{\alpha_i/2}}} \leq \max_{1\leq i \leq n} m_{n^{\alpha_i}} + (\log n)^{-\delta/3} \text{ a.s }\;.
\end{equation}
Now all we have to do is control $\max_{1\leq i \leq n} m_{n^{\alpha_i}}$, which is the maximum of a deterministic sequence.

Recall Vu's Theorem 1.4 in \cite{Vu07a}, applied to our situation where we are dealing with bounded random variables with mean 0 and variance 1: if the matrix $C$ has entries as above and is $n\times n$, then almost surely,
$$
\opnorm{\frac{C}{n^{1/2}}}\leq 2 + \kappa_0 \left(\frac{1-p}{p}\right)^{1/4}n^{-1/4}\log(n)\;,
$$
for some constant $\kappa_0$. So as soon as $(\log n)^{4}/(pn)$ remains bounded, so does $m_n$, the median of $\opnorm{\frac{C}{n^{1/2}}}$. In particular, if $(\log n)^{4}/(pn)\tendsto 0$, we have
$$
\limsup_{n\tendsto \infty} m_n\leq 2\;.
$$
Using elementary properties of the function $f$ such that $f(t)=(\log t)^4/t$, we can therefore conclude that if $\alpha_{\min}$ is such that
$$
\frac{\left(\alpha_{\min}\log n\right)^4}{n^{\alpha_{\min}}p}\tendsto 0\;,
$$
we have
$$
\limsup_{n\tendsto \infty}\max_{1\leq i \leq n} m_{n^{\alpha_i}} \leq 2\;.
$$
(Note that this is true because we are taking the maximum of elements of a fixed deterministic sequence that is asymptotically less than or equal to $2+\eps$, for any $\eps$ and the smallest argument is going to infinity. All the work using Talagrand's inequality was done to allow us to switch from having to control the maximum of a random sequence to that of a deterministic sequence.)

Now when $(\alpha_{min}\log n)^{4}/(pn^{\alpha_{\min}})\tendsto 0$, we have a fortiori $p n^{\alpha_{\min}}>(\log n)^{1+\delta}$ when $\alpha_{\min}>(\log n)^{(\delta-3)/4}$. So we conclude that when $(\alpha_{min}\log n)^{4}/(pn^{\alpha_{\min}})\tendsto 0$ and $\alpha_{\min}>(\log n)^{(\delta-3)/4}$,
$$
\limsup_{n\tendsto \infty} \max_{1\leq i \leq n} \opnorm{\frac{C_{\alpha_i}}{n^{\alpha_i/2}}} \leq 2 \text{ a.s }\;.
$$
\end{proof}

Let us now consider the related issue of understanding the matrix $E=r_p M\circ C$, where $r_p=\sqrt{(1-p)/p}$, $M$ is a deterministic matrix and $C$ is a random matrix as above. 

\begin{theorem}\label{thm:ControlOfE}
Suppose $E=r_p M\circ C$, where $C$ is a symmetric random matrix distributed as above, $M$ is a deterministic matrix and $r_p=\sqrt{(1-p)/p}$. Let us call $m_E$ a median of $\opnorm{E}$. Then we have
$$
P\left(|\opnorm{E}-m_E|>t\right)\leq 4 \exp\left(-\frac{p^2}{8\norm{M}_{\infty}^2}t^2\right)\;.
$$
Hence, in particular,
\begin{equation}\label{eq:Control2ndMomentOpNormErrorMatrix}
\Expect\left[\opnorm{E}^2\right]\leq m_E^2 +32\frac{\norm{M}_{\infty}^2}{p^2}+8m_E \sqrt{\frac{2\pi \norm{M}_{\infty}^2}{p^2}}\;.
\end{equation}
and
\begin{equation}\label{eq:Control3ndMomentOpNormErrorMatrix}
\Expect[\opnorm{E}^3]\leq 4 m_E^3+12\sqrt{\pi}\left(\frac{8\norm{M}_{\infty}^2}{p^2}\right)^{3/2}\;.
\end{equation}

\end{theorem}
\begin{proof}
The crux of the proof is quite similar to that of Theorem \ref{thm:ControlMaxOpNormsMatrices}: we will rely on Talagrand's concentration inequality for convex 1-Lipschitz functions of bounded random variables. To do so let us consider the map: $C\tendsto f(C)=\opnorm{M\circ C}$. This map $f$ is convex as the composition of a norm with an affine mapping. Let us now show that it is $(\sqrt{2}\norm{M}_{\infty})$-Lipschitz with respect to Euclidian norm: if we denote by $c^{(k)}_{i,j}$ the $(i,j)$-th entry of the matrix $C_k$, we have
\begin{align*}
\left|f(C_1)-f(C_2)\right|&=\left|\opnorm{M\circ C_1}-\opnorm{M\circ C_2}\right|\leq \opnorm{M\circ (C_1-C_2)}\\
&\leq \Frobnorm{M\circ (C_1-C_2)}=\sqrt{\sum_{i,j} M_{i,j}^2 (c^{(1)}_{i,j}-c^{(2)}_{i,j})^2}\\
&\leq \max_{i,j}|M_{i,j}| \sqrt{\sum_{i,j} (c^{(1)}_{i,j}-c^{(2)}_{i,j})^2}\leq \norm{M}_{\infty} \sqrt{2} \sqrt{\sum_{i\leq j} (c^{(1)}_{i,j}-c^{(2)}_{i,j})^2}
\end{align*}
Hence, $f$ is indeed a $(\sqrt{2}\norm{M}_{\infty})$-Lipschitz function of the entries of $C$ that are above or on the diagonal.
Now the function of $C$ we care about is $g(\cdot)=r_p f(\cdot)$, which is convex and $\sqrt{2}\norm{M}_{\infty}r_p$- Lipschitz.  Given that the entries of $C$ are bounded, we have, as in the proof of Theorem \ref{thm:ControlMaxOpNormsMatrices},
$$
P(|\opnorm{E}-m_E|>t)\leq 4 \exp\left(-\frac{p(1-p)}{8r_p^2\norm{M}_{\infty}^2}t^2\right)=4 \exp\left(-\frac{p^2}{8\norm{M}_{\infty}^2}t^2\right)\;.
$$
Now using the proof of Proposition 1.9 in \cite{ledoux2001} (see p.12 of this book), we conclude that
\begin{align*}
\Expect\left[|\opnorm{E}-m_E|\right]\leq 4 \sqrt{\frac{2\pi \norm{M}_{\infty}^2}{p^2}}\;, \text{ and }\\
\Expect\left[|\opnorm{E}-m_E|^2\right]\leq 32 \frac{\norm{M}_{\infty}^2}{p^2}\;.
\end{align*}
Therefore,
$$
\Expect\left[\opnorm{E}^2\right]\leq m_E^2 +32\frac{\norm{M}_{\infty}^2}{p^2}+8m_E \sqrt{\frac{2\pi \norm{M}_{\infty}^2}{p^2}}\;,
$$
since for $a$ and $b$ positive, $a^2\leq b^2+(a-b)^2+2b|a-b|$.

More generally, we see, using essentially Proposition 1.10 in \cite{ledoux2001} and elementary properties of the Gamma function,
that if the random variable $F$ is such that for a deterministic number $a_F$, $P(|F-a_F|>t)\leq C \exp(-cr^2)$, then
$$
\Expect[|F-a_F|^k]\leq C \Gamma(\frac{k}{2}+1)c^{-k/2}\;.
$$
Applying this result with $k=3$, we get
$$
\Expect\left[|\opnorm{E}-m_E|^3\right]\leq 3\sqrt{\pi}\left(\frac{8\norm{M}_{\infty}^2}{p^2}\right)^{3/2}\;.
$$
In our context, using the fact that, for positive $a$ and $b$, $(a+b)^3\leq 4(a^3+b^3)$ by convexity, we also have
$$
\Expect[\opnorm{E}^3]\leq 4 \left(m_E^3+3\sqrt{\pi}\left(\frac{8\norm{M}_{\infty}^2}{p^2}\right)^{3/2}\right)\;.
$$
\end{proof}

\subsection{Regularized eigenvector considerations}
We now have the following (regularized) second order accuracy result, which is a critical component of the proof of Theorem \ref{thm:SecondOrderAccuracyAveragingProc}, one of the main results of the paper.
\begin{theorem}\label{thm:SecondOrderAccuracyRegularizedAveragingProc}
Suppose that the assumptions of Theorem \ref{Thm:ControlOpNormErrorMatrix} are satisfied. We consider the approximation of $u$ the eigenvector associated with the largest eigenvalue of $M$. Recall that $v$ is the eigenvector corresponding to the leading eigenvalue of the subsampled matrix $S$. For $\eps>0$, we call $\tilde{v}_{\eps}$ the vector such that
$$
\tilde{v}_{\eps}=
\begin{cases}
v \text{ if } \opnorm{(\id+\Delta)^{-1}}\leq \frac{1}{\eps}\\
u-REu+\Delta RE u \text{ otherwise}
\end{cases}\;.
$$
Then, for any $\eta>0$, we have asymptotically,
$$
\norm{\Expect[u-\tilde{v}_{\eps}]}_2\leq \frac{8+\eta}{(\lambda_1-\lambda_2)^2} \frac{\mu^2}{pn^{\alpha_{\min}}}+\frac{16+\eta}{\eps(\lambda_1-\lambda_2)^3} \left(\frac{\mu^2}{pn^{\alpha_{\min}}}\right)^{3/2}\;.
$$
Suppose further that we are in an asymptotic setting where $\frac{1}{\lambda_1-\lambda_2}\frac{\mu}{(pn^{\alpha_{\min}})^{1/2}}\tendsto 0$. Then, $v-\tilde{v}_{\eps}=0$ with high-probability.
\end{theorem}
\begin{proof}
Let us first show that our regularization does not change the vector we are dealing with with high-probability.
$\tilde{v}_{\eps}=v$ as long as $\opnorm{(\id+\Delta)^{-1}}\leq 1/\eps$, which is guaranteed if $2\opnorm{E}/d\leq 1-\eps$. Since we assume that $\frac{1}{\lambda_1-\lambda_2}\frac{\mu}{(pn^{\alpha_{\min}})^{1/2}}\tendsto 0$ and we have according to Theorem \ref{thm:ControlOfE}
$\opnorm{E} \leq 2 \frac{\mu}{(pn^{\alpha_{\min}})^{1/2}}$ with high-probability, we conclude that with high-probability, $\tilde{v}_{\eps}=v$.

Using Equation \eqref{eq:ExactKthOrderExpansionOfEigenvector} with $j=1$, we see that, since $\opnorm{\Delta}\leq 2 \opnorm{R}\opnorm{E}$,
$$
\norm{\tilde{v}_{\eps}-(u-REu+\Delta RE u)}_2 \leq \frac{1}{\eps} \opnorm{\Delta}^2 \opnorm{RE}\leq \frac{4\opnorm{R}^3\opnorm{E}^3}{\eps}  \;.
$$
Recall that by construction $\Expect[E]=0$. Hence, since $R$ is a fixed deterministic matrix and $u$ is a deterministic vector,
$$
\Expect\left[\tilde{v}_{\eps}-u\right]=\Expect\left[\tilde{v}_{\eps}-u+REu\right]\;.
$$
So, if we now use the fact that $\norm{u}=1$, we have
\begin{align*}
\norm{\Expect\left[\tilde{v}_{\eps}-u\right]}_2&=\norm{\Expect\left[\tilde{v}_{\eps}-u+REu\right]}_2\\
&\leq \norm{\Expect\left[\tilde{v}_{\eps}-u+REu-\Delta REu\right]}_2+\norm{\Expect\left[\Delta REu\right]}_2\\
&\leq  \Expect\left[\norm{\tilde{v}_{\eps}-u+REu-\Delta REu}_2\right]+\Expect\left[\norm{\Delta REu}_2\right]\\
&\leq \Expect\left[\frac{4\opnorm{R}^3\opnorm{E}^3}{\eps}+2\opnorm{R}^2\opnorm{E}^2\right]\;.
\end{align*}


Let us now show that we can control the right-hand side of the previous equation.

We prove in Theorem \ref{thm:ControlOfE} that
$$
\Expect\left[\opnorm{E}^2\right]\leq m_E^2 +32\frac{\norm{M}_{\infty}^2}{p^2}+8m_E \sqrt{\frac{2\pi \norm{M}_{\infty}^2}{p^2}}\;,
$$
where $m_E$ is a median of the random variable $\opnorm{E}$. Our asymptotic control of $\opnorm{E}$ in \eqref{eq:norm-bound} gives allows us to control $m_E$, namely,
$$
\limsup_{n\tendsto \infty} m_E^2\leq 4 \frac{\mu^2}{pn^{\alpha_{\min}}}\;.
$$

In other respects, we clearly have $\norm{M}_{\infty}\leq \sum_{i=1}^n \lambda_i \norm{u_i}_{\infty}^2$, and hence
$$
\norm{M}_{\infty}\leq n^{-\alpha_{\min}} \mu\;.
$$
Hence,
$$
\frac{\norm{M}_{\infty}^2}{p^2}\leq  \frac{\mu^2}{\left(pn^{\alpha_{\min}}\right)^2}=o\left(\frac{\mu^2}{pn^{\alpha_{\min}}}\right)\;,
$$
since we are in a setting where $pn^{\alpha_{\min}}\tendsto \infty$. Similarly, $m_E\sqrt{\frac{\norm{M}_{\infty}^2}{p^2}}=o\left(\frac{\mu^2}{pn^{\alpha_{\min}}}\right)$, so we have for $\eta>0$,
$$
2\opnorm{R}^2\Expect\left[\opnorm{E}^2\right]\leq \frac{8+\eta}{(\lambda_1-\lambda_2)^2} \frac{\mu^2}{pn^{\alpha_{\min}}}
$$
asymptotically.

Furthermore, we prove in Theorem \ref{thm:ControlOfE} that
$$
\Expect[\opnorm{E}^3]\leq 4 m_E^3+12\sqrt{\pi}\left(\frac{8\norm{M}_{\infty}^2}{p^2}\right)^{3/2}\leq  4 m_E^3+\lo{\left(\frac{\mu^2}{pn^{\alpha_{\min}}}\right)^{3/2}}\;.
$$
Hence, for $\eta>0$,
$$
4\opnorm{R}^3 \Expect[\opnorm{E}^3] \leq \frac{16+\eta}{(\lambda_1-\lambda_2)^3} \left(\frac{\mu^2}{pn^{\alpha_{\min}}}\right)^{3/2}\;.
$$
\end{proof}

\subsection{On $\bm{\|C\|_2}$ when $\bm{p\ll (\log n)/n}$}\label{AppSubsec:TightnessResult}
At the end of Subsection \ref{ss:sparse-approx}, we mentioned a corollary (see below) of the following theorem:
\begin{theorem}\label{thm:LoseControlOfOpNormCForSmallp}
Suppose that  $p=(\log n)^{1-\delta} u_n/n$, for a fixed $\delta$ in $(0,1)$ and for a fixed $\kappa$, $0< u_n\leq \kappa$. Suppose further that we can find $v_n>0$ such that $v_n\tendsto \infty$, while $v_n= o(\log n,[u_n^{-1}(\log n)^{\delta}]^{1/4})$. Then
$$
\|C/\sqrt{n}\|_2 \tendsto \infty \text{ with probability one. }
$$
\end{theorem}
Recall that practically, this theorem suggests that if we don't sample enough the matrix $M$ (i.e $p$ is too small), a subsampling approximation to its eigenproperties is not likely to work. Let us now prove it.
\begin{proof}
Our strategy is to show that the largest diagonal entry of $C\trsp C/n$ goes to infinity. To do so, we will rely on results in random graph theory. Let us examine more closely this diagonal. Using the definition of $C$, we see that, if $T=C\trsp C$, and $d_i$ is the number of times $\sqrt{(1-p)/p}$ appears in the $i$-th column of $C$,
$$
T(i,i)=\frac{np}{1-p}+d_i \left(\frac{1-p}{p}-\frac{p}{1-p}\right)\;.
$$
Now $\{d_i\}$ is the degree sequence of an Erd\"os-Renyi random graph. According to \cite{BollobasRandomGraphs}, Theorem 3.1, if $k$ is such that $n \binom{n-1}{p} p^k (1-p)^{n-1-k}\tendsto \infty$, then, if $X_k$ is the number of vertices with degree greater than $k$,
$$
\lim_{n\tendsto \infty} P(X_k \geq t)=1\;,
$$
for any $t$. So if we can exhibit such a $k$, then $\max d_i \geq k$ with probability going to 1. We now note that for small $p$,
$$
\left(\frac{1-p}{p}-\frac{p}{1-p}\right)\geq \frac{1}{2p}\;.
$$
Hence, if our $k$ is also such that $k/pn\tendsto \infty$, we will indeed have
$$
\max_i \frac{T(i,i)}{n}\tendsto \infty
$$
and the theorem will be proved.\\
We propose to take $k=np(1+v_n)$. According to \cite{BollobasRandomGraphs}, Theorem 1.5, if $h=k-np$, and $q=1-p$,
\begin{equation}\label{eq:KeyEqThmExplosionOpNormC}
\binom{n}{p} p^k (1-p)^{n-k}\geq \frac{1}{\sqrt{2\pi p q n}} \exp\left(-\frac{h^2}{2pqn}-\frac{h^3}{2q^2 n^2}-\frac{h^4}{3p^3n^3}-\frac{h}{pn}-\beta\right)\;,
\end{equation}
where $\beta=1/(12 k)+1/(12(n-k))$. In our case, $h=np v_n$. Let us show that all the terms in the exponential are negligible compared to $\log n$ as $n\tendsto \infty$:
\begin{itemize}
\item $\beta\tendsto 0$ because $k\tendsto \infty$ and $npv_n=\lo{(\log n)^{2-\delta}}$, given that $v_n=\lo{\log n}$. Hence $n-k\tendsto \infty$. 
\item $h/(pn)=v_n=\lo{\log n}$ by assumption.
\item $h^4/(pn)^3=np v_n^4=\lo{u_n (\log n)^{1-\delta}(\log n)^{\delta}/u_n}=\lo{\log n}$, since $v_n=\lo{(u_n^{-1}(\log n)^{\delta})^{1/4}}$.
\item $h^3/n^2=npv_n^3 p^2=\lo{np v_n^4 p^2}=\lo{p^2 \log n}$, since $v_n^3=\lo{v_n^4}$ ($v_n\tendsto \infty$ by assumption).
\item $h^2/np=np v_n^2=\lo{np v_n^4}=\lo{\log n}$.
\end{itemize}
In light of these estimates, we have as $n\tendsto \infty$,
$$
\sqrt{n}\exp\left(-\frac{h^2}{2pqn}-\frac{h^3}{2q^2 n^2}-\frac{h^4}{3p^3n^3}-\frac{h}{pn}-\beta\right)\tendsto \infty\;.
$$
Therefore, with this choice of $k$,
$$
n \binom{n-1}{p} p^k (1-p)^{n-1-k}\tendsto \infty\;.
$$
We can finally conclude that
$$
\max_i T(i,i)/n \geq \frac{k}{2np} \text{ with probability going to 1}\;.
$$
But because $v_n\tendsto \infty$, we have $k/(2np)\tendsto \infty$ and the theorem is proved.
\end{proof}
We have the following corollary to which we appealed in Subsection \ref{ss:sparse-approx}.
\begin{corollary}
When $p\sim (\log n)^{1-\delta}/n$ for some fixed $\delta \in (0,1)$,
$$
\|C/\sqrt{n}\|_2 \tendsto \infty \text{ with probability one. }
$$
\end{corollary}
The previous corollary follows immediately from Theorem \ref{thm:LoseControlOfOpNormCForSmallp}, by noticing that  $u_n$ is lower bounded under our assumptions and by taking $v_n=(\log n)^{\delta/5}$. 

\subsection{Variance computations}\label{ss:AppendixVarianceComputations}
We provide some details here to complement the explanations we gave in the proof of Theorem \ref{thm:VarianceREu} in Subsection \ref{ss:variance}.
\paragraph{On $\Expect[E^2]$} Let us explain why this matrix is diagonal and compute the coefficients on the diagonal. Recall that $E=\sqrt{(1-p)/p}M\circ C$, where $C$ is a random matrix whose above-diagonal elements are independent, have mean 0 and variance 1. $E$ is naturally symmetric and we call $E_i$ its $i$-th column. Naturally, $E^2(i,j)=E_i\trsp E_j$. Suppose first that $i\neq j$. The elements of $E_i$ and $E_j$ are independent, except for $E_{ij}$ and $E_{ji}$, which are equal. In particular, $E_{ki}$ and $E_{kj}$ are independent for all $1\leq k\leq n$. Recall also that $\Expect[C]=0$, so $\Expect[E]=0$. Combining all these elements, we conclude that, if $i\neq j$,
$$
\Expect[E_i\trsp E_j]=\sum_{k=1}^n \Expect[E_{ki} E_{kj}]=\sum_{k=1}^n \Expect[E_{ki}]\Expect[E_{kj}]=0\;.
$$
Therefore $\Expect[E^2]$ is diagonal. Let us now turn our attention to computing the elements of the diagonal. This is simple since
$$
\Expect[E_i\trsp E_i]=\frac{1-p}{p}\sum_{k=1}^n M_{ki}^2 \Expect[E_{ki}^2]=\frac{1-p}{p}\sum_{k=1}^n M_{ki}^2=\frac{1-p}{p}\norm{M_i}_2^2\;.
$$
We note that this is the result we announced in the proof of Theorem \ref{thm:VarianceREu} in Subsection \ref{ss:variance}.

\paragraph{On $\var(u\trsp E u)$} 
Rewriting this quantity as a sum of independent quantities greatly simplifies the computation. If we pursue this route, we have
$$
u\trsp E u=\sum_{i,j} u(i) u(j) E_{ij}=2 \sum_{i>j} u(i) u(j) E_{ij}+\sum_{i} u(i)^2 E_{ii}\;.
$$
Because the previous expression is a sum of independent random variables, we immediately conclude that
\begin{align*}
\frac{p}{1-p} \var(u\trsp E u)&= 4 \sum_{i>j} u(i)^2 u(j)^2 M^2_{ij}+\sum_{i}u(i)^4 M^2_{ii}\\
&=2(2\sum_{i>j} u(i)^2 u(j)^2 M^2_{ij}+\sum_{i}u(i)^4 M^2_{ii})-\sum_{i}u(i)^4 M^2_{ii}\;.
\end{align*}
Calling $w=u\circ u$ and ${\cal M}=M\circ M$, we immediately recognize in the last expression the quantity
$$
2(w\trsp {\cal M} w)-\sum_k w(k)^2 {\cal M}_{kk}\;,
$$
as announced in the proof of Theorem \ref{thm:VarianceREu}.

\bibliographystyle{alpha}
{\small \bibliography{research,MainPerso}}

\end{document}